\documentclass[reqno, 10pt]{amsart}
\usepackage{amssymb,amsmath,amsthm,}
\usepackage{a4wide}
\usepackage{amscd}
\usepackage{amsfonts}
\usepackage{amssymb}
\usepackage{latexsym}
\usepackage{color}
\usepackage{esint}
\usepackage{graphicx}
\usepackage{caption}
\usepackage{subcaption}
\usepackage{amsthm}
\usepackage[colorlinks]{hyperref}
\usepackage[makeroom]{cancel}
\usepackage{mathtools}

\let\hide\iffalse
\let\unhide\fi

\newtheorem{theorem}{Theorem}[section]

\newtheorem{lemma}[theorem]{Lemma}

\let\p=\partial

\newcommand{\R}{\mathbb{R}}

\renewcommand{\S}{\mathbb{S}}
\newcommand{\T}{\mathbb{T}}
\newcommand{\be}{\begin{equation}}
\newcommand{\bm}{\begin{multline}}
\newcommand{\ee}{\end{equation}}

\newcommand{\Bes}{\begin{eqnarray*}}
	\newcommand{\Ees}{\end{eqnarray*}}
\newcommand{\Be}{\begin{equation}}
\newcommand{\Ee}{\end{equation}}

 \numberwithin{equation}{section}

\def\p{\partial}

\def\R{\mathbb{R}}

\def\B{\begin{equation}}
\def\E{\end{equation}}
\def\BN{\begin{eqnarray*}}
\def\EN{\end{eqnarray*}}

\usepackage{color}

\begin{document}
	\date{ \today
	}
	
	\title{Global Stability of the Boltzmann equation for a polyatomic gas with initial data allowing large oscillations}
	
	\author[G. Ko]{Gyounghun Ko}
	\address[GK]{Center for Mathematical Machine Learning and its Applications(CM2LA), Department of Mathematics, Pohang University of Science and Technology, South Korea}
	\email{gyounghun347@postech.ac.kr}
	
	\author[S. Son]{Sung-jun Son}
	\address[SS]{Center for Mathematical Machine Learning and its Applications(CM2LA), Department of Mathematics, Pohang University of Science and Technology, South Korea}
	\email{sungjun129@postech.ac.kr}

			\begin{abstract}
			In this paper, we consider the Boltzmann equation for a polyatomic gas. We establish that the mild solution to the Boltzmann equation on the torus is globally well-posed, provided the initial data that satisfy bounded velocity-weighted $L^{\infty}$ norm and the smallness condition on the initial relative entropy. Furthermore, we also study the asymptotic behavior of solutions, converging to the global Maxwellian with an exponential rate. A key point in the proof is to develop the pointwise estimate on the gain term of non-linear collision operator for Gr\"{o}nwall's argument. 
		\end{abstract}
	\maketitle
	\tableofcontents

	\section{Introduction}

\textit{The Boltzmann equation} is a fundamental kinetic model for describing collisional particles. The equation provides a statistical description of the dynamic behavior of a dilute gas, capturing the time evolution of the molecular distribution function. 
 The Boltzmann equation for monoatomic gases can occasionally be insufficient for dealing with intricate scenarios such as satellite orbits, contamination, or atmospheric re-entry. Hence, various models, notably the Borgnakke-Larsen model proposed in \cite{BL}, have been introduced to cater to polyatomic gases. These models incorporate internal energy to account for rotational and vibrational motion. The Borgnakke-Larsen model has garnered widespread usage in mathematical research, accommodating non-integer degrees of freedom per particle.
	Although alternative models, see \cite{Be,Bi,MBKK}, have been proposed, the computational expense associated with complex degrees of freedom remains prohibitive. Simplified models with one degree of freedom that have emerged to alleviate computational costs can be found in \cite{LP}. Based on these, Pullin and Kuscer have explored phenomena such as internal energy changes in polyatomic gases. In \cite{BDLP}, Bourgat et al. generalized the model and demonstrated an H-theorem related to the second law of thermodynamics. Additionally, Monte Carlo simulations were conducted to validate the results. Furthermore, for inelastic and elastic collisions, simplified models can be found in \cite{Br,EP}.\\
	\indent In the paper, we concerned with the Boltzmann equation for polyatomic gases: 
		\begin{align} \label{poly be}
		\begin{split}
			\p_{t}F + v\cdot\nabla F =Q(F,F),& \cr
			\hspace{1cm}F(0,x,v,I)=F_0(x,v,I),&
		\end{split}
	\end{align}
	where $F=F(t,x,v,I)$ is the probability density function of polyatomic gases with position $x\in\mathbb{T}^3$, velocity $v\in\R^3$, and internal energy $I \in \R_+$ at time $t\in[0,\infty)$. Here, the internal energy $I$ is related to rotational and vibrational motion . The collision operator $Q(F,G)$ in \eqref{poly be} is defined by
	\begin{align}\label{BLCO}
		\begin{split}
			Q(F,G)&:=\int_{\mathbb{S}^2\times[0,1]^2\times\R^3 \times \R_{+}} B\left(\frac{F'G'_{*}}{(I'I'_{*})^{\delta/2-1}}-\frac{FG_{*}}{(II_{*})^{\delta/2-1}}\right)\cr
			&\quad \times\left(r(1-r)\right)^{\delta/2-1}(1-R)^{\delta-1}R^{1/2}(II_*)^{\delta/2-1}dwdRdrdv_*dI_*,
		\end{split}
	\end{align}
	where we denote the following abbreviations
	\begin{align*}
		F:=F(t,x,v,I), \; G_{*}:=G(t,x,v_{*},I_{*}), \; F':=F(t,x,v',I'), \; \text{and } G'_{*}:=G(t,x,v'_{*},I'_{*}).
	\end{align*}	
	In \eqref{BLCO}, $B=B(v,v_*,I,I_*)$  is  the transition function, throughout this paper, we consider that as the total energy in the relative velocity-center of mass velocity framework with constant angular function , see \cite{GP}
	\begin{align*}
		B:=C\left(\frac{|v-v_*|^2}{4}+I+I_*\right)^\frac{2-\alpha}{2},
	\end{align*}
	where $\alpha\in [0,2)$. In this model, the number of degrees of freedom $\delta\geq 2$ is introduced, representing the relation of specific heat considered at the hydrodynamic Navier-Stokes level, resulting in the following relationship:
\begin{align*}
			\gamma=\frac{\delta+5}{\delta+3}.
	\end{align*} 
\noindent Under an elastic collision, two particles with velocities and internal energies $(v, I)$ and $(v_*, I_*)$ in phase space $\R^3 \times \R_+$ change to $(v', I')$ and $(v'_*, I'_*)$ after the collision. Post-collision velocities and internal energies satisfy the conservation laws of momentum and total energy as follows:
	\begin{align}\label{CL}
		\begin{split}
			v+v_*&=v'+v'_*,\cr
			\frac{|v|^2}{2}+\frac{|v_*|^2}{2}+I+I_*&=\frac{|v'|^2}{2}+\frac{|v'_*|^2}{2}+I'+I'_*.
		\end{split}
	\end{align}
	For convenience, we use notation as below:
	\begin{align}\label{NOTA}
		V:= \frac{|v+v_*|}{2}, \quad V' := \frac{|v' +v'_*|}{2}, \quad u:= v-v_*, \quad \text{and,} \quad u' :=v'-v'_*. 
	\end{align}
	Thus, from the momentum and energy conservation laws in \eqref{CL}, one obtains
	\begin{align*}
		V=V', \quad		\frac{|u|^2}{4}+I+I_*&=\frac{|u'|^2}{4}+I'+I'_* . 
	\end{align*} 
	The collision operator is based on the Borgnakke-Larsen procedure, which states that at post-collision total energy $E$ is distributed into pure kinetic energy and pure internal energy (Figure \ref{fig1}). This is represented by introducing the parameter $R\in[0,1]$ such that 
	\begin{align*}
		\frac{|u'|^2}{4}=RE, \quad I'+I'_*=(1-R)E.  
	\end{align*}
	Additionally, a parameter $r\in[0,1]$ is introduced to distribute the the proportion of total internal energy
		 \begin{align*}
		I'=r(1-R)E , \quad I'_*=(1-r)(1-R)E.
	\end{align*}
	
	\begin{figure}[h]
	\centering
	\includegraphics[scale=0.5]{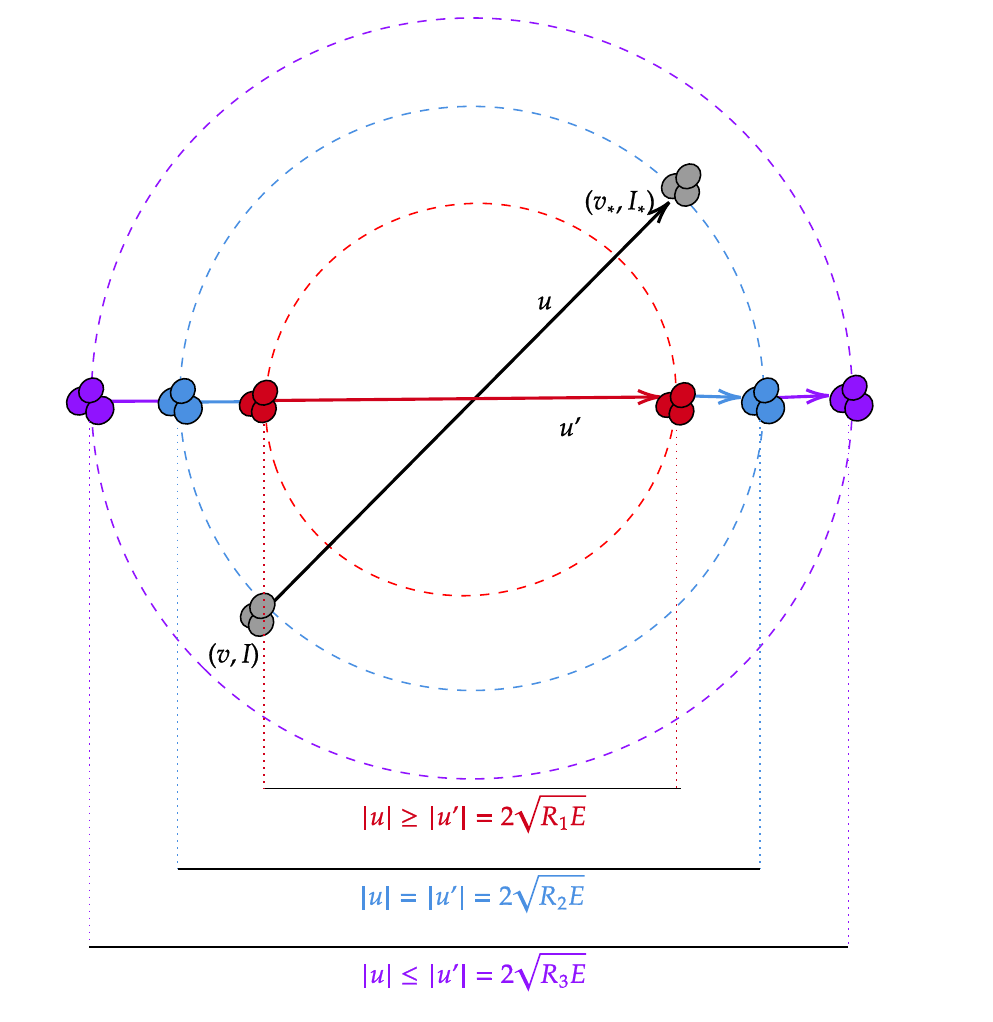}  
	\caption{This figure describes the changing ratio of kinetic energy to internal energy for two polyatomic molecules after a collision. To aid readers' understanding, the phenomenon of increasing internal energy and decreasing kinetic energy ($R_1 < R_2 < R_3$) is illustrated purely from the perspective of kinetic energy.}
	 \label{fig1}
\end{figure}
	\noindent In this paper, we also consider an equivalent form of the collision operator given by \eqref{BLCO}, which was first proposed in \cite{BN2}:
		\begin{align}\label{CO}
			\begin{split}
				Q(F,G)&:=\int_{(\R^3 \times \R_{+})^3} W(v,v_{*},I,I_{*} \vert v',v'_{*},I',I'_{*})\left(\frac{F'G'_{*}}{(I'I'_{*})^{\delta/2-1}}-\frac{FG_{*}}{(II_{*})^{\delta/2-1}}\right) dv_{*}dv'dv'_{*}dI_{*}dI'dI'_{*}\\
				&=Q_+(F,G) -Q_-(F,G).
			\end{split}
		\end{align}
	Here, $W(v,v_{*},I,I_{*} \vert v',v'_{*},I',I'_{*})$ represents the transition probabilities for polyatomic gases: 
		\begin{align*}
			W(v,v_{*},I,I_{*} \vert v',v'_{*},I',I'_{*}) & :=4(II_*)^{\delta/2}\sigma\left(|v-v_*|,\left|\frac{v-v_*}{|v-v_*|}\cdot\frac{v'-v'_*}{|v'-v'_*|}\right|,I,I_*,I',I'_*\right)\frac{|v-v_*|}{|v'-v'_*|}\cr
			&\quad\quad\times \delta_{3} (v+v_{*}-v'-v'_{*}) \delta_1 \left(\frac{\vert v \vert^2}{2} + \frac{ \vert v_{*}\vert^2}{2} +I+I_{*} -\frac{\vert v' \vert^2}{2}-\frac{\vert v'_{*}\vert^2}{2} -I'-I'_{*}\right),
		\end{align*}
		where the scattering cross sections $\sigma$ is given by
		\begin{align*}
			\sigma:=C\frac{\sqrt{\vert u \vert^2+4I+4I_* -4I' -4I'_{*}}}{\vert u \vert\Phi^{\delta +(\alpha-1)/2}}(I'I'_*)^{\delta/2-1}.
		\end{align*}
		Through cancellation by \eqref{CL}, the form of $W(v,v_{*},I,I_{*} \vert v',v'_{*},I',I'_{*})$ is reduced as follows:
		\begin{align} \label{W}
			\begin{split}
				&W(v,v_{*},I,I_{*} \vert v',v'_{*},I',I'_{*})\cr
				&= \frac{4C(II_{*}I'I'_{*})^{\delta/2-1}}{\Phi^{\delta +(\alpha-1)/2}}\delta_{3} (v+v_{*}-v'-v'_{*}) \delta_1 \left(\frac{\vert v \vert^2}{2} + \frac{ \vert v_{*}\vert^2}{2} +I+I_{*} -\frac{\vert v' \vert^2}{2}-\frac{\vert v'_{*}\vert^2}{2} -I'-I'_{*}\right),
			\end{split}
		\end{align}
		where $\Phi$ and $\Phi'$ are defined by 
		\begin{align*}
			\begin{split}
				\Phi=\frac{\vert u \vert^2}{4}+I+I_{*}\quad \text{and}\quad \Phi' =\frac{\vert u' \vert^2}{4} +I'+I'_{*}.
			\end{split}
		\end{align*} 
		The equivalent form can be derived through the following change of variables
		\begin{align*}
			R=\frac{ |v'-v'_*|}{4\Phi},\quad r=\frac{I'}{(1-R)\Phi}, \quad \text{and} \quad  \Phi'=\frac{|u'|^2}{4}+I'+I'_*. 
		\end{align*}
		For more details on this, we refer to \cite{BN2}.\\ 
\indent In the Boltzmann equation for polyatomic gases, the global Maxwellian $M$ satisfying $Q(M,M)=0$ is given by 
\begin{align} \label{mu}
	M = M(v,I) := \frac{I^{\delta/2-1}}{(2\pi)^{3/2} \Gamma(\delta/2)} e^{-\vert v \vert^2/2-I},
\end{align}
where $ \Gamma(s)$ is the  gamma function defined by $ \Gamma(s) := \int_0^{\infty} x^{s-1} e^{-x} dx$. Due to the following collision invariant property
\begin{align} \label{collisional inv}
	\int_{\R^3 \times \R_+} Q(F,F)	\left(\begin{array}{c}1\cr v\cr \frac{|v|^2}{2}+I\end{array}\right) dv dI =0, 
\end{align}
the solution $F(t,x,v,I)$ to the Boltzmann equation for polyatomic gases satisfies the conservation laws of mass, momentum, and total energy, see \cite{BN2}: 
\begin{align} \label{laws}
	\begin{split}
		\left(\begin{array}{c}
			M_I \cr J_I \cr E_I 
		\end{array}\right)
		&:=\int_{\T^3 \times \R^3 \times \R_+}  (F(t,x,v,I)-M(v,I)) 	\left(\begin{array}{c}1\cr v\cr \frac{|v|^2}{2}+I\end{array}\right) dxdvdI \\
		&= 	\int_{\T^3 \times \R^3 \times \R_+}  (F_0(x,v,I)-M(v,I))	\left(\begin{array}{c}1\cr v\cr \frac{|v|^2}{2}+I\end{array}\right) dxdvdI.
	\end{split}
\end{align}
Moreover, the entropy dissipation for \eqref{poly be} can be obtained by \cite{BBBD,BDLP,GP}, known as the $H$-theorem, 
\begin{align*}
	\frac{d}{dt}\int_{\mathbb{T}^3\times\R^3\times\R_+}F(t,x,v,I)\ln F(t,x,v,I)dxdvdI\leq0.
\end{align*}
In the Boltzmann equation for monoatomic gases, there have been studies on its well-posedness theories. DiPerna and Lions \cite{Diperna-Lion} proved the global existence of renormalized solutions for general initial data with finite mass, energy, and entropy. However, the uniqueness of these solutions remains undetermined.  On the other hand, asymptotic behavior of solutions is one of the main subjects for the Boltzmann equation due to the $H$-theorem. We may expect the convergence to the global Maxwellian since the entropy dissipation vanishes at the global Maxwellian. For this issue, Desvillettes and Villani \cite{DV} obtained convergence to the global Maxwellian with an almost exponential rate under global-in-time a priori regularity conditions and a Gaussian lower bound for the solutions.    \\
\indent At the same time, well-posedness theories and asymptotic behavior have been studied based on perturbation framework. When the initial data are sufficiently close to the global Maxwellian, Ukai \cite{Ukai} first constructed the global-in-time solutions to the Boltzmann equation in a periodic box $\T^3$. Moreover, under the perturbation setup, Guo \cite{GuoVPB,GuoVMB} used the energy method in high Sobolev spaces to construct classical solutions without a physical boundary. However, within bounded domains, high-order regularity of solutions may not be expected; see \cite{GKTT2017,Kim2011,KLKRM}. Thus, Guo \cite{Guo10} developed a novel $L^2$-$L^\infty$ bootstrap argument to construct low regularity solutions with certain boundary conditions. Subsequently, under the specular reflection boundary condition, Kim-Lee \cite{KimLee} replaced the analytic condition in \cite{Guo10} with general $C^3$ uniform convex bounded domains. Additionally, for more recent works concerning the boundary value problems, see references \cite{CKLVPB,KimLeeNonconvex,KL_holder,KKL}.   \\ 
\indent Recent mathematical research on the Boltzmann equation for polyatomic gases has been actively ongoing. For instance, in a space homogeneous setting, Gamba and Pavi$\acute{c}$- $\check{C}$oli$\acute{c}$ established the existence and uniqueness for the non-linear case, considering hard potentials for relative speed and internal energy under the extended Grad assumption, see \cite{GP}.  On the other side, in the space inhomogeneous case, the study on the linearized collision operator in the perturbation framework near the global Maxwellian can be found in \cite{BN2}. The coercivity estimate of the collision frequency and compactness of the integral operator for the monatomic case were derived in \cite{DH,glassey}. On the other hand, for polyatomic gases, the properties of the linearized operator corresponding to the collision operator \eqref{BLCO} were discussed in \cite{BST1, BST2} under additional assumptions on the transition function. In \cite{BN2}, they assumed the form \eqref{CO} of the collision operator and dealt with the coercivity and compactness of the linearized operator. Furthermore, in \cite{BN1}, we can find results for the discrete internal energy model. Interested readers may refer to \cite{CSC,EG} for related models. For the replaced model with a relaxational operator like the BGK model, interested readers refer to \cite{BCRY,BS,HRY,PY2,PY1,SY,Yun}. \\ 
\indent In those previous works related to the perturbation theory, the difference between initial data and the global Maxwellian must be sufficiently small. To the best of our knowledge, well-posedness theory of the Boltzmann equation with large-data class of initial data is still unsolved. Recently, \cite{DHWY2017} obtained the large-amplitude solution to the Boltzmann equation without smallness of $L^\infty$ norm. However, they need an additional smallness assumption on the relative entropy and $L^1_x L^\infty_v$ for initial data. For the boundary value problems and  types of models, there are results that have constructed large-amplitude solutions in a similar sense; see \cite{BKLY_BGK,CaoJFA,DKL2023,DW2019,KLP2022}. \\ 
\indent In this paper, in a periodic box $\T^3$, we construct large-amplitude solutions to the Boltzmann equation modeling polyatomic gases where only the initial relative entropy is small, relaxing the condition of smallness of the $L^\infty$ norm of the initial data assumed in \cite{DLpoly}. Moreover, we obtain the convergence to the global Maxwellian with an exponential rate.\\
\indent Let $F(t,x,v,I)= M(v,I) + \sqrt{M(v,I)} f(t,x,v,I)$. Then, the Boltzmann equation \eqref{poly be} is reformulated as 
\begin{align} \label{re BE}
	\p_t f + v\cdot \nabla_x f + Lf = \Gamma(f,f).
\end{align}
The linear operator $L$ is defined by 
\begin{align*}
	Lf &:= -\frac{1}{\sqrt{M}}\left(Q(M,\sqrt{M}f) +Q(\sqrt{M}f,M)\right)=\nu(v,I) f - Kf,
\end{align*} 
where the collision frequency $\nu$ and the integral operator $K$ are given by 
\begin{align}
	\nu(v,I)&:= \frac{1}{\sqrt{M}}Q_{-}(\sqrt{M},M)=\int_{(\R^3 \times \R_{+})^3} W(v,v_{*},I,I_{*} \vert v',v'_{*},I',I'_{*})\frac{M_*}{(II_*)^{\delta/2-1}} dv_* dv' dv'_* dI_* dI' dI'_*,  \label{collision frequency}\\
	Kf&:= \frac{1}{\sqrt{M}}\left(Q(M,\sqrt{M}f)  + Q_+ (\sqrt{M}f,M)  \right) \nonumber \\
	&= \frac{1}{\sqrt{M}}\int_{(\R^3 \times \R_{+})^3} W(v,v_{*},I,I_{*} \vert v',v'_{*},I',I'_{*})\left (\frac{M'\sqrt{M'_*}}{(I'I'_*)^{\delta/2-1}}f'_* +(\frac{M'_*\sqrt{M'}}{(I'I'_*)^{\delta/2-1}}f' \right)dv_* dv' dv'_* dI_* dI' dI'_* \nonumber \\
	& \quad -\int_{(\R^3 \times \R_{+})^3} W(v,v_{*},I,I_{*} \vert v',v'_{*},I',I'_{*})\frac{\sqrt{MM_*}}{(II_*)^{\delta/2-1}}f_* dv_* dv' dv'_* dI_* dI' dI'_* \nonumber \\
	&=K_2 f - K_1 f.  \label{Kf}
\end{align}
Note that the collision frequency $\nu(v,I)$ is equivalent to $(1+|v| + \sqrt{I})^{2-\alpha}$; see \cite{DLpoly}. In other words, there exist positive constants $C_0,C_1,$ and $\nu_0$ such that 
\begin{align*}
	\nu_0 \leq C_0(1+|v| + \sqrt{I})^{2-\alpha} \leq \nu(v,I) \leq C_1 (1+|v| + \sqrt{I})^{2-\alpha}, \quad \alpha \in [0,2). 
\end{align*}
The non-linear term $\Gamma(f,f) = \Gamma_+ (f,f) -\Gamma_- (f,f)$ is given by 
\begin{align*}
	\Gamma_+ (f,f) = \frac{1} {\sqrt{M}} Q_+(\sqrt{M}f, \sqrt{M}f) , \quad \Gamma_- (f,f) = \frac{1}{\sqrt{M}} Q_- (\sqrt{M}f , \sqrt{M}f).
\end{align*}
\subsection{Main results and scheme of the proof}
\begin{theorem} \label{main thm}
	Define the velocity weight function
	\begin{align} \label{weight}
		w(v,I): = (1+ \vert v \vert + \sqrt{I})^{\beta}, \quad \beta > 7. 
	\end{align}
	Then, for any $M_0>0$, there exists $\varepsilon = \varepsilon(M_0)>0$ such that if initial data satisfy that $F_0 (x,v,I) = M(v,I)+ \sqrt{M(v,I)} f_0 (x,v,I) \geq 0$, the conservation laws \eqref{laws} with $(M_I,J_I,E_I)=(0,0,0)$, and 
	\begin{align*}
		\| wf_0 \|_{L^{\infty}} \leq M_0, \quad \mathcal{E}(F_0) \leq \varepsilon, 
	\end{align*}
	then the initial value problem \eqref{poly be} on the polyatmoic Boltzmann equation admits a unique global-in-time mild solution $F(t,x,v,I)= M(v,I) + \sqrt{M(v,I)} f(t,x,v,I) \geq 0 $ satisfying 
	\begin{align*}
		\| wf(t) \|_{L^{\infty}} \leq C_{M_0} e^{-\vartheta t}, \quad \forall t \in [0,\infty),
	\end{align*}
	where $\vartheta>0$ is a positive constant. 
\end{theorem}
Now, let us introduce the scheme of the proof of the above main theorem. For the large-amplitude problem, it is hard to construct the global $L^{\infty}$ bound for solutions due to the non-linearity. In particular, although the nonlinear term $\Gamma$ could be treated as (see \cite{DLpoly})
\begin{align*}
	| \Gamma (f,f)(t,x,v,I) | \lesssim \nu(v,I) || f (t) ||_{\infty}^2, 
\end{align*}
we may fail to obtain the $L^{\infty}$ boundedness for solutions without the smallness assumption on the $L^{\infty}$ norm. To overcome this difficulty, we develop the pointwise estimate on the nonlinear term $\Gamma_+$ (see Lemma \ref{GE}) inspired by the nonlinear estimate for monatomic gases; see references \cite{DHWY2017,DKL2023,DW2019,KLP2022}. Using the change of variables and careful analysis for the internal energy used in \cite{BN2,DLpoly}, we obtain the nonlinear estimate 
\begin{align*}
	|w(v) \Gamma_+ (f,f)(t,x,v,I)| \lesssim ||wf(t)||_{L^{\infty}} \left( \int_{\R^3 \times \R_+} (1+|v_*| + \sqrt{I_*})^{-2\beta+8} |wf(t,x,v_*,I_*)|^2 dv_* dI_*\right)^{1/2}.
\end{align*} 
To apply the above estimate, we derived a new operator $\mathcal{R}(f)$ combining with $\nu(v,I)(f)$ and $\Gamma_-(f,f)$, leaving only the nonlinear term $\Gamma_+(f,f)$ in the source term. In other words, the operator $\mathcal{R}(f)$ is defined by 
\begin{align*}
	\mathcal{R}(f)  :=  \int_{(\R^3 \times \R_+)^3} W(v,v_*,I,I_*| v',v'_*,I',I'_*) \left( M_* + \sqrt{M_*} f _*\right)  dv_*dv'dv'_* dI_* dI'dI'_* . 
\end{align*}
 Whenever $\mathcal{R}(f)$ also has a uniform positive lower bound like the collision frequency (see Lemma \ref{Rf est}), we can derive the Gr\"{o}nwall inequality through double Duhamel iteration as follows: 
 \begin{align*}
 	|w(v,I)f(t,x,v,I)| &\lesssim C_{f_0} \left(1+ \int_0^t || wf(s) ||_{L^{\infty}} ds \right)e^{-t}\\
 	&\quad +C(\bar{M}) \int_0^t \int_0^s e^{-(t-s)}e^{-(s-\tau)}\iint_{|\eta| \leq N, |\eta_*|\leq N, L \leq N, L_* \leq N} | wf(\tau,X(\tau), \eta_*, L_*)|,
 \end{align*} 
 where $X(\tau):= x- v(t-s)- \eta (s-\tau)$ and $\bar{M}$ is a priori bound $\sup_{0\leq t \leq T_0} ||wf(t)||_{L^{\infty}} \leq \bar{M}$. 
 Using the change of variables and controlling $L^p$ by initial relative entropy (Lemma \ref{relative entropy lemma}), we close a priori estimate under the smallness condition on the initial relative entropy. \\
 
\noindent \textbf{Notations.} Throughout this paper, $\Vert \cdot \Vert_{L^{\infty}}$ denotes the $L^{\infty}(\T^3_{x} \times \R^3_{v})$-norm. For convenience of notation, we have used the abbreviation $f=f(t,x,v,I)$, $f_*= f(t,x,v_*,I_*)$, $f' = f(t,x,v',I')$ and $f'_* = f(t,x,v'_*, I'_*)$. Also, for the global Maxwellian $M$, we have abbreviated $M =M(v,I)$, $M_* = M(v_*, I_*)$, $M' = (v', I')$ and $M'_* = M(v'_*, I'_*)$. 
\section{Preliminaries}
\subsection{The operator $K$} 
\begin{lemma} \cite{DLpoly} \label{K esti}
	Let $K=K_2-K_1$ be defined in \eqref{Kf}. There exist kernels $k_i = k_i(v,v_*,I,I_*)$ such that 
	\begin{align*}
		K_i f(t,x,v,I) = \int_{\R^3 \times \R_+} k_i (v,v_*,I,I_*) f(t,x,v_*,I_*) dv_* dI_*, \quad i =1,2.  
	\end{align*}
	If we define the kernel $k= k_2 - k_1$, the velocity weight function  $w=w(v,I) = (1+|v| + \sqrt{I})^{\beta}$ in \eqref{weight}, and 
	\begin{align*}
		k_w = k_w (v,v_*,I,I_*) =  k(v,v_*,I,I_*) \frac{w(v,I)}{w(v_*,I_*)},
	\end{align*}
	then we have the following estimate
	\begin{align} \label{kw esti}
		\int_{\R^3 \times \R_+} k_w (v,v_*,I,I_*) e^{\varepsilon \vert v-v_* \vert^2} (1+I_*)^m dv_* dI_* \leq \frac{C_{\varepsilon,m}}{1+ \vert v \vert +I^{1/4}},
	\end{align}
	for $0\leq \varepsilon \leq 1/64$ and $ 0 \leq m \leq 1/8$. Moreover, we have 
	\begin{align} \label{kw_l2}
		\sup_{(v,I) \in \R^3 \times \R_+} \int_{\R^3 \times \R_+} |k_w (v,v_*,I,I_*)|^2 dv_* dI_* < \infty. 
	\end{align}
	\end{lemma}
	\begin{proof}
	Since the detailed proof of the kernel $k_w$ estimate \eqref{kw esti} is provided in \cite{DLpoly}, we now prove \eqref{kw_l2} which means $L^2$ boundedness of the kernel $k_w$.  From the definition of the operator $K$ in \eqref{Kf}, the kernel $k_1$ is expressed by
	\begin{align} \label{k1 esti}
		k_1(v,v_*,I,I_*) = \int_{(\R^3)^2 \times (\R_+)^2} W(v,v_*,I,I_* | v',v'_* , I' ,I'_*) \frac{\sqrt{MM_*}}{(II_*)^{\delta/2-1}} dv'dv'_* dI' dI'_* .
	\end{align}
	Using \eqref{W}, we directly deduce that 
	\begin{align}\label{k1 esti2}
	\begin{split}
		k_1(v,v_*,I,I_*)  &\leq C \int_{(\R^3)^2 \times (\R_+)^2} e^{-\frac{|v|^2}{4}-\frac{|v_*|^2}{4}-\frac{I}{2}-\frac{I_*}{2}} \frac{1}{\Phi^{\delta +(\alpha-1)/2}}(II_{*})^{\delta/4-1/2}(I'I'_{*})^{\delta/2-1} \\ 
		&\quad\times \delta_{3} (v+v_{*}-v'-v'_{*}) \delta_1 \left(\frac{\vert v \vert^2}{2} + \frac{ \vert v_{*}\vert^2}{2} +I+I_{*} -\frac{\vert v' \vert^2}{2}-\frac{\vert v'_{*}\vert^2}{2} -I'-I'_{*}\right)dv'dv'_*dI'dI'_*,
	\end{split}
\end{align}
	where we used $\sqrt{\vert u \vert^2+4I+4I_* -4I' -4I'_{*}}= \vert u' \vert $ under \eqref{CL}.
	Moreover, we notice that $\Phi = \frac{|u|^2}{4}+I+I_* = \frac{|u'|^2}{4}+I'+I_*' = \Phi'$. Thus, we have the lower bound for the term $\Phi^\delta$ in the right-hand side of \eqref{k1 esti2}: 
	\begin{align} \label{Phi esti} 
		\Phi^{\delta} = \Phi^{\frac{\delta}{4}-\frac{1}{4}} (\Phi')^{\frac{3}{4}\delta +\frac{1}{4}}=\left(\frac{|u|^2}{4}+I+I_*\right)^{\frac{\delta}{4}-\frac{1}{4}} \left(\frac{|u'|^2}{4}+I'+I'_*\right)^{\frac{3}{4}\delta + \frac{1}{4} } \geq (I_*)^{\frac{\delta}{4}-\frac{1}{4}} (I'I'_*)^{\frac{3}{8}\delta+\frac{1}{8}}.
	\end{align}
	If we use the notation $V= v+v_*$ and $V' = v'+v'_*$, then it follows from \eqref{k1 esti} and \eqref{Phi esti} that 
	\begin{align*}
		k_1 (v,v_*,I,I_*) &\leq C \Phi^{(2-\alpha)/2} I^{\delta/4-1/2}(I_*)^{-1/4}e^{-\frac{|v|^2}{8}-\frac{|v_*|^2}{8}-\frac{I}{4}-\frac{I_*}{4}} \int_{(\R^3)^2\times (\R_+)^2} e^{-\frac{|v'|^2}{8}-\frac{|v_*'|^2}{8}-\frac{I'}{4}-\frac{I_*'}{4}} (I'I'_*)^{\delta/8-9/8}\\
		&\quad \times \delta_{3} (V-V') \delta_1 \left(\sqrt{\vert u \vert^2+4I+4I_* -4I' -4I'_{*}}- \vert u' \vert \right) dv'dv'_*dI'dI'_*\\
		&\leq C (I_*)^{-1/4} e^{-\frac{|v|^2}{16}-\frac{|v_*|^2}{16}-\frac{I}{8}-\frac{I_*}{8}},
	\end{align*}
	where we used $\frac{\delta}{8} -\frac{9}{8}>-1$ for $\delta \in [2,\infty)$. Hence, we obtain 
	\begin{align} \label{k1}
		\int_{\R^3 \times \R_+} \left | k_1(v,v_*,I,I_*)\frac{w(v,I)}{w(v_*,I_*)}\right |^2 dv_* dI_* &\leq C \int_{\R^3 \times \R_+} (I_*)^{-1/2} e^{-\frac{|v|^2}{8}-\frac{|v_*|^2}{8}-\frac{I}{4}-\frac{I_*}{4}} \frac{w^2(v,I)}{w^2(v_*,I_*)} dv_* dI_* \leq C.	
	\end{align}
	For $k_2$, we refer to \cite{DLpoly} for the following estimate: 
	\begin{align*}
		k_2(v,v_*,I,I_*) \frac{w(v,I)}{w(v_*,I_*)} \leq \frac{C}{|u|} e^{-\frac{|u|^2}{16}}\int_{(\R_+)^2} e^{-\frac{I'}{4}-\frac{I'_*}{4}}(II_*)^{\delta/4-1/2}(I'I'_*)^{\delta/2-1} \phi(I,I_*,I',I'_*) dI'dI'_*,
	\end{align*}
	where 
	\begin{align*}
		\phi(I,I_*,I',I'_*) &:= \frac{1}{(II_*)^{\delta/4-1/2}(I'I'_*)^{\delta/4+1/4}} \chi_{\{I\leq 1, I_*\leq 1\}}+\frac{1}{I^{\delta/4-1/2}I_*^{3\delta/4}} \chi_{\{I\leq 1, I_*\geq 1\}}\\
		&\quad+\frac{1}{I^{3\delta/4}I_*^{\delta/4-1/2}} \chi_{\{I\geq 1, I_*\leq 1\}}+\frac{1}{I^{3\delta/4-5/4}I_*^{\delta/4+3/4}} \chi_{\{I\geq 1, I_*\geq 1\}}.
		\end{align*}
		Therefore, one obtains that 
		\begin{align} \label{k2 esti} 
			\begin{split}
			&\int_{\R^3 \times \R_+} \left | k_2(v,v_*,I,I_*) \frac{w(v,I)}{w(v_*,I_*)}\right |^2 dv_* dI_* \\
			&\leq C \int_{\R^3 \times \R_+} \frac{e^{\frac{-|v-v_*|^2}{8}}}{|v-v_*|^2}  \left(\int_{(\R_+)^2} e^{-\frac{I'}{4}-\frac{I'_*}{4}}(II_*)^{\delta/4-1/2}(I'I'_*)^{\delta/2-1} \phi(I,I_*,I',I'_*) dI'dI'_*\right)^2dv_* dI_* . 
			\end{split}
		\end{align}
			To further estimate $k_2$, we compute the following integral in \eqref{k2 esti} using $\phi$: 
		\begin{align}\label{k2 integral}
			\begin{split}
				&\int_{(\R_+)^2} e^{-\frac{I'}{4}-\frac{I'_*}{4}}(II_*)^{\delta/4-1/2}(I'I'_*)^{\delta/2-1} \phi(I,I_*,I',I'_*) dI'dI'_*\\
				&\leq C  \int_{(\R_+)^2} e^{-\frac{I'}{4}-\frac{I'_*}{4}} \left( (I'I'_*)^{\delta/4-5/4} \chi_{\{I\leq 1, I_*\leq 1\}} + I_*^{-\delta/2-1/2}(I'I'_*)^{\delta/2-1} \chi_{\{I\leq 1, I_*\geq 1\}} \right. \\
				& \quad \left. + I^{-\delta/2-1/2}(I'I'_*)^{\delta/2-1} \chi_{\{I\geq 1, I_*\leq 1\}}+I^{-\delta/2+3/4}I_*^{-5/4}(I'I'_*)^{\delta/2-1} \chi_{\{I\geq 1, I_*\geq 1\}} \right) dI'dI'_*  \\
				&\leq C \left(1\chi_{\{I_* \leq 1\}} + I_*^{-5/4} \chi_{\{I_* \geq 1\}}\right).
			\end{split}
		\end{align}
		where the last inequality comes from $-\delta/2-1/2<-5/4$. From \eqref{k2 esti} and \eqref{k2 integral}, we have
		\begin{align} \label{k2} 
			\begin{split}
				\int_{\R^3 \times \R_+} \left | k_2(v,v_*,I,I_*) \frac{w(v,I)}{w(v_*,I_*)}\right |^2 dv_* dI_* &\leq C \int_{\R^3 \times \R_+} \frac{e^{\frac{-|v-v_*|^2}{8}}}{|v-v_*|^2} \left[1\chi_{\{I_* \leq 1\}} + I_*^{-5/4} \chi_{\{I_* \geq 1\}}\right]^{2} dv_* dI_*\\
				& \leq C \int _{\R^3}   \frac{e^{\frac{-|v-v_*|^2}{8}}}{|v-v_*|^2} dv_* \\
				&\leq C.
			\end{split} 
		\end{align}
		Combining \eqref{k1} and \eqref{k2}, we get the desired result
		\begin{align*}
			\int_{\R^3 \times \R_+} |k_w (v,v_*,I,I_*)|^2 dv_* dI_* \leq C, 
		\end{align*} 
		where the constant $C>0$ does not depend on $v$ and $I$. This completes the proof. 
	\end{proof}
	
\subsection{Pointwise estimate on non-linear terms}

\begin{lemma}\label{GE}
There exists a constant $C>0$ such that 
\begin{align} \label{gain esti}
	\vert w(v,I) \Gamma_+ (f,f)(t,x,v,I) \vert \leq \frac{C \Vert wf(t) \Vert_{L^{\infty} }}{1+ \vert v \vert +I^{1/4}} \left(\int_{\R^3 \times \R_+} (1+\vert v_* \vert + \sqrt{I_*})^{-2\beta +8} \vert wf(t,x,v_*,I_*) \vert^2 dv_* dI_* \right)^{1/2}, 
\end{align}
where $w(v,I) = (1+ \vert v \vert + \sqrt{I})^{\beta}$ for $\beta > 7$. 
\end{lemma}
\begin{proof}
By the energy conservation law 
\begin{align*}
	\frac{\vert v \vert^2}{2} + \frac{ \vert v_{*}\vert^2}{2} +I+I_{*} = \frac{\vert v' \vert^2}{2} + \frac{ \vert v'_{*}\vert^2}{2} +I'+I'_{*},
\end{align*}
we have 
\begin{align*}
	\frac{\vert v \vert^2}{2} + I \leq \frac{\vert v' \vert^2}{2} + \frac{ \vert v'_{*}\vert^2}{2} +I'+I'_{*},
\end{align*}
which implies 
\begin{align*}
	w(v,I) \leq 4^{\beta}w(v',I') \quad \text{or} \quad w(v,I) \leq 4^{\beta}w(v'_*, I'_*).
\end{align*}
Then, it follows that
\begin{align} \label{esti 1}
	\begin{split}
		&\vert w(v,I) \Gamma_+ (f,f)(v,I) \vert \cr
		&= \frac{w(v,I)}{\sqrt{M(v,I)}} \vert Q_+ (\sqrt{M}f ,\sqrt{M}f) \vert \\
		&\leq \frac{4^{\beta}}{\sqrt{M(v,I)}} \int_{(\R^3 \times \R_+)^3} W(v,v_{*},I,I_{*} \vert v',v'_{*},I',I'_{*})\frac{\sqrt{M'}\sqrt{M'_*}}{(I'I'_*)^{\delta/2-1}} \vert w' f' \cdot f'_*\vert dv_* dv' dv'_* dI_* dI' dI'_* \\
		&\quad + \frac{4^{\beta}}{\sqrt{M(v,I)}} \int_{(\R^3 \times \R_+)^3} W(v,v_{*},I,I_{*} \vert v',v'_{*},I',I'_{*})\frac{\sqrt{M'}\sqrt{M'_*}}{(I'I'_*)^{\delta/2-1}} \vert w'_* f'_* \cdot f'\vert dv_* dv' dv'_* dI_* dI' dI'_* \\
		&\leq 4^{\beta} \Vert wf \Vert_{L^{\infty} } \int_{(\R^3 \times \R_+)^3} W(v,v_{*},I,I_{*} \vert v',v'_{*},I',I'_{*})\frac{\sqrt{M_*}}{(II_*I'I'_*)^{\delta/4-1/2}} \vert f'_* \vert dv_* dv' dv'_* dI_* dI' dI'_* \\ 
		&\quad + 4^{\beta} \Vert wf \Vert_{L^{\infty} } \int_{(\R^3 \times \R_+)^3} W(v,v_{*},I,I_{*} \vert v',v'_{*},I',I'_{*})\frac{\sqrt{M_*}}{(II_*I'I'_*)^{\delta/4-1/2}}  \vert f'\vert dv_* dv' dv'_* dI_* dI' dI'_*\\ 
		&:= I_1 + I_2, 
	\end{split} 
\end{align}
where we used 
\begin{align*}
	\sqrt{M_*} = \frac{ (II_*)^{\delta/4-1/2}}{ (I'I'_*)^{\delta/4-1/2}} \frac{\sqrt{M'}\sqrt{M'_*}}{\sqrt{M}}.
\end{align*}
For $I_2$ in \eqref{esti 1}, rename $(v',I')$ to $(v_*, I_*)$ to get 
\begin{align} \label{I2 est}
	I_2 = 4^{\beta} \Vert wf \Vert_{L^{\infty} }\int_{(\R^3 \times \R_+)^3} W(v,v',I,I' \vert v_*,v'_{*},I_*,I'_{*})\frac{\sqrt{M'}}{(II_*I'I'_*)^{\delta/4-1/2}}  \vert f_* \vert dv_{*}dv'dv'_{*}dI_{*}dI'dI'_{*}.
\end{align}
To estimate $I_2$, we denote 
\begin{align}
	\xi := v- v' , \quad \xi_* := v_* - v'_*, \quad \Psi := \frac{ \vert \xi \vert^2}{4} + I+ I', \quad \zeta := (v_* -v') \cdot \frac{u}{\vert u \vert}, \quad \text{and} \quad \Delta J := I_* +I'_* -I-I'.
\end{align}
\begin{figure}[h]
	\centering
	\includegraphics[scale=0.65]{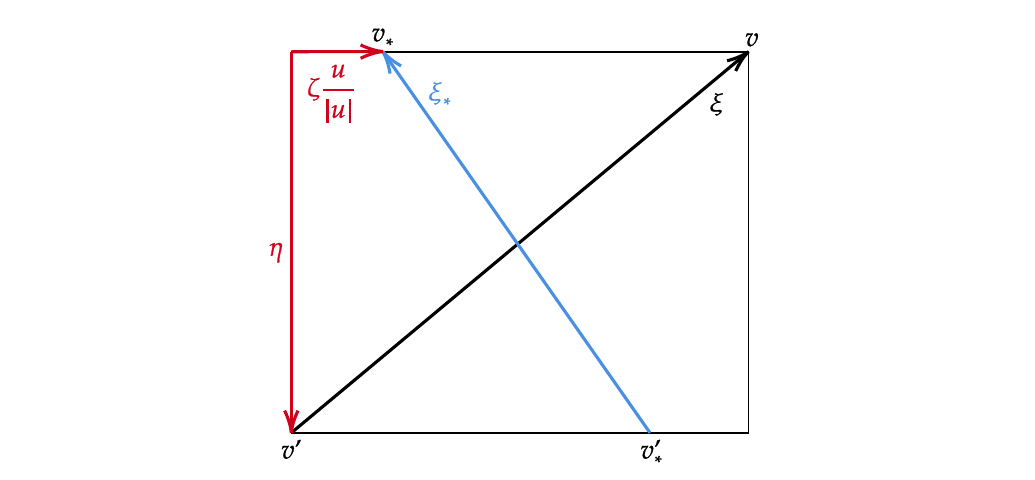}  
	\caption{Representation of variables}
	\label{fig2}
\end{figure}

\noindent By the definition \eqref{W} of $W$, we obtain 
\begin{align} \label{W1}
		\begin{split}
	W(v,v',I,I' | v_* , v'_*, I_*, I'_*) &= C (II_*I'I'_*)^{\delta/2-1}\frac{ \chi_{\{\vert \xi \vert^2+4I+4I'>4I_* +4I'_*\}}}{\Psi^{\delta+ (\alpha-1)/2}} \cr
		&\quad\times \delta_1 \left(\frac{\vert v \vert^2}{2} + \frac{ \vert v'\vert^2}{2} +I+I' -\frac{\vert v_* \vert^2}{2}-\frac{\vert v'_{*}\vert^2}{2} -I_*-I'_{*}\right) \delta_3(v+v'-v_*-v'_*).
		\end{split}
\end{align}
Here, using properties of the Dirac delta, we get (Figure \ref{fig2})
	\begin{align*}
		v'_*=v+v'-v_*,
	\end{align*} 
	so that
	\begin{align}\label{v'*}
		|v'_*|^2=|v|^2+|v'|^2+|v_*|^2+2v\cdot v'-2v'\cdot v_*-2v_*\cdot v.
	\end{align}
	Inserting \eqref{v'*} into $\delta_1$, we have
	\begin{align*}
		&\delta_1 \left(\frac{\vert v \vert^2}{2} + \frac{ \vert v'\vert^2}{2} +I+I' -\frac{\vert v_* \vert^2}{2}-\frac{\vert v'_{*}\vert^2}{2} -I_*-I'_{*}\right)\cr
		&=\delta_1 \left(v'\cdot v_*+v_*\cdot v-v\cdot v'-|v_*|^2 +I+I' -I_*-I'_{*}\right)\cr
		&=\delta_1 \left((v_*-v')\cdot(v-v_*)-\Delta J \right)\cr
		&=\frac{1}{|u|}\delta_1 (\zeta -\frac{\Delta J}{ \vert u \vert}).
	\end{align*}
Thus, we rewrite $W(v,v',I,I' | v_* , v'_*, I_*, I'_*)$ as follows:
	\begin{align}\label{W2}
		W(v,v',I,I' | v_* , v'_*, I_*, I'_*)=C \frac{(II_*I'I'_*)^{\delta/2-1}}{ \vert u \vert} \frac{ \chi_{\{\vert \xi \vert^2+4I+4I'>4I_* +4I'_*\}}}{\Psi^{\delta+ (\alpha-1)/2}} \delta_1 (\zeta -\frac{\Delta J}{ \vert u \vert}) \delta_3( u+ u').
	\end{align}
We resolve 
\begin{align*}
	v' - v_* = \eta - \zeta n, \quad n:= \frac{u}{ \vert u \vert},
\end{align*}
It holds that $dv' dv'_* = du' d \zeta d \eta$. It follows from \eqref{I2 est} and \eqref{W2} that 
\begin{align} \label{I2 est 2}
	\begin{split}
	I_2 &= C_{\beta} \Vert wf \Vert _{L^{\infty} } \int_{\R^3 \times \R_+} \int_{\R^3 \times \R_+ \times (\R^3)^{\perp n} \times (\R_+)^2 } e^{-\frac{\vert v' \vert^2}{4} - \frac{I'}{2}} \frac{(II_* I_*')^{\delta/4-1/2} (I')^{\delta/2-1}}{\vert u \vert} \\
	&\quad \times \frac{\chi_{\{\vert \xi \vert^2+4I+4I'>4I_* +4I'_*\}}}{\Psi^{\delta+ (\alpha-1)/2}} \delta_1 (\zeta -\frac{\Delta J}{ \vert u \vert}) \delta_3( u+ u') du' d\zeta d \eta dI' dI'_* \vert f(v_*,I_*) \vert dv_* dI_*.
	\end{split}
\end{align}
Note that 
\begin{align} \label{cov}
	\begin{split}
	\vert v' \vert ^2 &= \vert v_*- \zeta n + \eta\vert^2 \\
	&= \vert (v_*)_{||} - \zeta n +(v_*)_{\perp} + \eta \vert^2 \\
	&= \vert  v_*\cdot n - \zeta \vert^2 + \vert (v_*)_{\perp} + \eta \vert^2,
	\end{split}
\end{align}	
where 
\begin{align*}
	(v_*)_{||} := (v_* \cdot n ) n, \quad (v_*)_{\perp} := v_* - (v_*)_{||}.
\end{align*}
Since $\Psi = \frac{\vert \xi \vert^2}{4} + I  + I' = \frac{\vert \xi' \vert^2}{4} + I_* + I'_*$ by $\delta_1 ( \zeta -\frac{\Delta J}{\vert u \vert}) \delta_3 ( u+u')$, one obtains that 
\begin{align} \label{esti 2} 
	\Psi = \frac{\vert \xi \vert^2}{4} + I  + I'  \geq C \vert \xi \vert \sqrt{ \frac{\vert \xi \vert^2}{4} + I  + I'-I_* - I'_*} = C \vert \xi \vert \vert \xi _* \vert \geq C \vert \eta \vert^2, 
\end{align}
and 
\begin{align*}
	\Psi ^{\delta-1/2} \geq \phi_i (I,I_*, I' , I'_*), \quad i =1,2,3,4, 
\end{align*}
where 
\begin{align*}
	\phi_1 (I,I_*,I',I'_*) &:= (II_*)^{ \delta/4-1/2} (I'I'_*)^{ \delta/4+1/4}, \\ 
	\phi_2 (I,I_*,I',I'_*) &:= I_*^{\delta/4-1/2} (I'_*)^{3\delta/4},\\
	\phi_3 (I,I_*,I',I'_*) &:= I^{\delta/2+1/2}(I_*I'_*)^{\delta/4-1/2} ,\\
	\phi_4 (I,I_*,I',I'_*) &:= I^{\delta/2-3/4} I_*^{\delta/4-1/2}( I'_*)^{\delta/4+3/4}.
\end{align*}
We set 
\begin{align*}
	\phi(I,I_*,I',I'_*) &:= \frac{1}{\phi_1 (I,I_*,I',I'_*)} \chi_{\{ I \leq 1, I'_* \leq 1\}} + \frac{1}{\phi_2 (I,I_*,I',I'_*)} \chi_{\{I\leq 1 I'_* \geq 1\}}\\
	& \quad +\frac{1}{\phi_3 (I,I_*,I',I'_*)} \chi_{\{I \geq 1, I'_* \leq 1\}}+ \frac{1}{\phi_4 (I,I_*,I',I'_*)} \chi_{\{I \geq 1, I'_* \geq 1\}},
\end{align*}
so that 
\begin{align} \label{psi esti}
	\frac{1}{\Psi^{\delta -1/2}} \leq \phi(I,I_*, I' ,I'_*).
\end{align}
Thus, from \eqref{I2 est 2}, \eqref{cov}, \eqref{esti 2}, and \eqref{psi esti}, we obtain 
\begin{align*}
	 &\int_{\R^3 \times \R_+ \times (\R^3)^{\perp n} \times (\R_+)^2 } e^{-\frac{\vert v' \vert^2}{4} - \frac{I'}{2}} \frac{(II_* I_*')^{\delta/4-1/2} (I')^{\delta/2-1}}{\vert u \vert} \frac{\chi_{\{\vert \xi \vert^2+4I+4I'>4I_* +4I'_*\}}}{\Psi^{\delta+ (\alpha-1)/2}}\\
	 &\quad \times \delta_1 (\zeta -\frac{\Delta J}{ \vert u \vert}) \delta_3( u+ u') du' d\zeta d \eta dI' dI'_*\\
	 &\leq \frac{C}{\vert u \vert} \int_{\R^3 \times \R_+ \times (\R^3)^{\perp n} \times (\R_+)^2 } e^{-\vert v_* \cdot n - \zeta  \vert^2}e^{-\frac{I'}{2}} (II_* I'_*)^{\delta/4 -1/2} (I')^{\delta/2 -1} \\
	 &\quad \times \phi (I,I_*,I',I'_*) \frac{e^{-\vert (v_*)_{\perp} +\eta\vert^2}}{\vert \eta \vert^{\alpha}}\delta_1 (\zeta - \frac{\Delta J}{\vert u \vert}) \delta_3 (u+u') du' d\zeta d\eta dI' dI'_*\\ 
	 &\leq \frac{C}{\vert u \vert} \int_{(\R_+)^2}e^{-\frac{I'}{2}} (II_* I'_*)^{\delta/4 -1/2} (I')^{\delta/2 -1} \phi(I,I_*,I',I'_*) dI' dI'_*.
\end{align*}
Hence, we can further bound $I_2$ as 
\begin{align*}
	I_2 &\leq C_{\beta} \Vert wf \Vert_{L^{\infty}\  } \int_{\R^3 \times (\R_+)^3} \frac{1}{\vert v-v_* \vert} e^{-\frac{I'}{2}} (II_*I'_*)^{\delta/4-1/2} (I')^{\delta/2-1} \phi(I,I_*,I',I'_*) \vert f(v_*, I_*)\vert dI_* dI' dI'_* dv_*\\
	&\leq C_{\beta} \Vert wf \Vert_{L^{\infty}\  } \int_{\R^3} \frac{1}{\vert v- v_* \vert} \\
	&\quad \times \int_{(\R_+)^3} (I')^{\delta/2 -1} e^{-\frac{I'}{2}} (II_*I'_*)^{\delta/4 -1/2} \phi(I,I_*,I',I'_*) (1+ \vert v_* \vert + \sqrt{I_*})^{-\beta} \vert wf(v_*,I_*) \vert dI' dI'_* dI_* dv_* \\
	& \leq C_{\beta}  \Vert wf \Vert_{L^{\infty}\  } \int_{\R^3} \frac{1}{\vert v-v_* \vert} \int_{(\R_+)^2} \left( (I'_*)^{-3/4} \chi_{\{I \leq 1, I'_* \leq 1\}} + (I'_*)^{-\delta/2-1/2}  \chi_{\{I\leq 1, I'_* \geq 1\}}\right.\\
	&\left. \quad + I^{-\delta/4-1}  \chi_{\{I \geq 1, I'_*\leq 1\}} + I^{-\delta/4+1/4}  (I'_*)^{-5/4} \chi_{\{I \geq 1, I'_* \geq 1\}}\right) (1+ \vert v_* \vert +\sqrt{I_*})^{-\beta} \vert wf(v_*,I_*)\vert dI_* dI'_* dv_*\\
	&\leq  C_{\beta} \frac{\Vert wf \Vert_{L^{\infty}\  } }{1+I^{1/4}}\int_{\R^3 \times \R_+} \frac{1}{ \vert v -v_* \vert} \frac{1}{(1+ \vert v_* \vert)^2} \frac{1}{(1+ I_*)} (1+ \vert v_* \vert + \sqrt{I_*})^{-\beta +  4}  \vert wf(v_*,I_*) \vert dI_* dv_* \\
	& \leq  C_{\beta}  \frac{\Vert wf \Vert_{L^{\infty}\  }}{1+I^{1/4}} \left(\int_{\R^3 \times \R_+} \frac{1}{\vert v-v_* \vert^2}\frac{1}{(1+\vert v_* \vert)^4}\frac{1}{(1+I_*)^2}dI_* dv_* \right)^{1/2}\\
	&\quad \times \left(\int_{\R^3 \times \R_+} (1+ \vert v_* \vert + \sqrt{I_*})^{-2\beta +   8} \vert wf(v_*,I_*) \vert^2 dI_* dv_* \right)^{1/2}\\
	&\leq  C_{\beta}  \frac{\Vert wf \Vert_{L^{\infty}\  }}{1+\vert v \vert + I^{1/4}} \left(\int_{\R^3 \times \R_+} (1+ \vert v_* \vert + \sqrt{I_*})^{-2\beta +   8}\vert wf(v_*,I_*) \vert^2 dI_* dv_* \right)^{1/2},
\end{align*}
where we have used the fact
\begin{align*}
	\int_{\R^3} \frac{1}{\vert v-v_* \vert^2} \frac{1}{(1+\vert v_* \vert)^4} dv_* \leq \frac{1}{(1+ \vert v \vert)^2}.
\end{align*}
For $I_1$, by renaming $(v_*', I_*')$ to $(v_*,I_*)$, $I_1$ can be rewritten in a similar form as \eqref{I2 est} for $I_2$. Consequently, we can obtain the same estimate for $I_1$. Thus, we derive the estimate \eqref{gain esti}, which completes the proof. 
\end{proof}
\begin{lemma} \label{nonlinear esti} \cite{DLpoly}
There exists a constant $C>0$ such that 
\begin{align*}
	| w(v,I) \Gamma(f,f) (t,x,v,I) | \leq C \nu(v,I) ||wf(t)||_{L^{\infty}}^2, 
\end{align*}
where $w(v,I) = (1+ |v| + \sqrt{I})^{\beta}$ for $\beta>7$. 
\end{lemma} 
\begin{proof}
	Since it is the same as Lemma 2.3 in \cite{DLpoly}, we omit the detailed proof. 
\end{proof}
\subsection{Relative entropy} 
\begin{lemma} \label{relative entropy lemma}
	Assume that $F(t,x,v,I)$ is a solution to the Boltzmann equation \eqref{poly be}. For any $t \geq 0$, we have 
	\begin{align*}
		\int_{\mathbb{T}^3 \times \R^3 \times \R_+} \frac{1}{4M} \vert F-M  \vert^2 \chi_{\{\vert F-M \vert \leq M\}} dx dv dI + \int_{\T^3 \times \R^3 \times \R_+} \frac{1}{4} \vert F-M \vert \chi_{\{\vert F-M \vert > M\}} dxdvdI \leq \mathcal{E}(F_0), 
	\end{align*}
	where the relative entropy $\mathcal{E}(F)$ is defined by 
	\begin{align} \label{relative entropy}
		\mathcal{E}(F) : = \int_{\T^3 \times \R^3 \times \R_+} \left( \frac{F}{M} \ln \frac{F}{M} - \frac{F}{M} +1 \right) M dxdvdI. 
	\end{align}
\end{lemma}
\begin{proof}
	It follows from Taylor's expansion that 
	\begin{align*}
		F \ln \left(I^{1-\delta/2} F\right) -M \ln \left(I^{1-\delta/2}M\right) = \left(1+ \ln \left(I^{1-\delta/2} M\right)\right)(F-M) + \frac{1}{2\tilde{F}} \vert F-M \vert^2,
	\end{align*} 
	where $\tilde{F}$ is between $F$ and $M$. Then, we compute 
	\begin{align*}
		\frac{1}{2\tilde{F}} \vert F-M \vert^2 = F \ln \left(I^{1-\delta/2} F\right) - M \ln \left(I^{1-\delta/2}M\right) - \left(1+ \ln \left(I^{1-\delta/2} M\right)\right)(F-M)=\psi \left( \frac{F}{M} \right) M, 
	\end{align*}
	where $\psi(x) = x \ln x - x +1$. Thus, 
	\begin{align} \label{relative esti} 
		\int_{\T^3 \times \R^3 \times \R_+} 	\frac{1}{2\tilde{F}} \vert F-M \vert^2 dxdvdI  = \int_{\T^3 \times \R^3 \times \R_+} \psi  \left( \frac{F}{M} \right) Mdx dv dI . 
	\end{align}
	For the left-hand side in \eqref{relative esti}, we divide 
	\begin{align*}
		1 = 1 \chi_{\{\vert F- M \vert \leq M\}} + 1 \chi_{\{\vert F- M \vert > M\}} .
	\end{align*}
	On $\{ \vert F - M \vert > M \}$, we have 
	\begin{align*}
		\frac{ \vert F -M  \vert}{\tilde{F}}  = \frac{F -M} {\tilde{F}} > \frac{F - \frac{1}{2} F}{F} = \frac{1}{2},
	\end{align*} 
	where we used $F> 2M$. On the other hand, over $\{ \vert F- M \vert \leq M \}$, we have $0\leq F \leq 2M$, which implies that 
	\begin{align*}
		\frac{1}{ \tilde {F}} \geq \frac{1}{2M}. 
	\end{align*}
	From \eqref{relative esti}, one obtains that 
	\begin{align*}
		&\int_{\mathbb{T}^3 \times \R^3 \times \R_+} \frac{1}{4M} \vert F-M  \vert^2 \chi_{\{\vert F-M \vert \leq M\}} dx dv dI + \int_{\T^3 \times \R^3 \times \R_+} \frac{1}{4} \vert F- M \vert \chi_{\{ \vert F-M \vert > M\}} dxdvdI \\
		&\leq  \int_{\T^3 \times \R^3 \times \R_+} \psi  \left( \frac{F}{M} \right) M dx dv dI .
	\end{align*}
	Due to $\psi' (x) = \ln x $, we can deduce from \eqref{poly be} that 
	\begin{align*}
		\p_t \left[ \psi \left( \frac{F}{M}\right) M \right] + \nabla_x \cdot  \left[\psi \left( \frac{F}{M} \right)M v\right] = Q(F,F) \ln \frac{F}{M} = Q(F,F) \ln \frac{I^{1-\delta/2} F} {I^{1-\delta/2} M}.
	\end{align*}
	Taking integration over $(x,v,I) \in \T^3 \times \R^3 \times \R_+ $, it follows from the collision invariant property \eqref{collisional inv} that
	\begin{align*}
		\frac{d}{dt} \int_{\mathbb{T}^3 \times \R^3 \times \R_+} \psi \left( \frac{F}{M} \right) M dxdvdI = \int_{\mathbb{T}^3 \times \R^3 \times \R_+} Q(F,F)\ln (I^{1-\delta/2} F) dxdvdI. 
	\end{align*}
	Because of the following inequality
	\begin{align*}
		\int_{\mathbb{T}^3 \times \R^3 \times \R_+} Q(F,F)\ln (I^{1-\delta/2} F) dxdvdI\leq 0, 
	\end{align*}
	we get 
	\begin{align} \label{relative esti 2}
		\int_{\mathbb{T}^3 \times \R^3 \times \R_+} \psi \left( \frac{F}{M} \right) M dxdvdI \leq \int_{\mathbb{T}^3 \times \R^3 \times \R_+} \psi \left( \frac{F_0}{M} \right) M dxdvdI =\mathcal{E}(F_0). 
	\end{align}
	Combining \eqref{relative esti} and \eqref{relative esti 2} yields that 
	\begin{align*}
		\int_{\mathbb{T}^3 \times \R^3 \times \R_+} \frac{1}{4M} \vert F-M  \vert^2 \chi_{\{\vert F-M \vert \leq M\}} dxdvdI + \int_{\T^3 \times \R^3 \times \R_+} \frac{1}{4} \vert F- M \vert \chi_{ \{\vert F-M \vert > M\}} dxdvdI \leq \mathcal{E}(F_0).
	\end{align*} 
\end{proof}
\section{A priori estimate}

In this section, we obtain the lower bound for a new operator $\mathcal{R}(f)$ and $L^\infty$ estimate for solutions $h(t,x,v,I)=w(v,I)f(t,x,v,I)$ under a priori assumption \eqref{HC}. To verify this, we rewrite the Boltzmann equation in \eqref{re BE} by introducing the new operator $\mathcal{R}(f)$, which is derived from combining $\Gamma_-(f,f)$ and $\nu(v,I)f$ in the following manner:
\begin{align*}
	\mathcal{R}(f)f:= \frac{1}{\sqrt{M}} \left [ Q_- (\sqrt{M}f, \sqrt{M}f) +Q_- (\sqrt{M}f, M)\right] = \Gamma_-(f,f) +\nu(v,I)f.
	\end{align*}
 The equation can be express by
\begin{align*}
	\p_t f + v\cdot \nabla_x f + \mathcal{R}(f) f =Kf+ \Gamma_+ (f,f).
\end{align*} 
To obtain $L^{\infty}$ estimate from the reformulated Boltzmann equation above, we check that the new operator $\mathcal{R}(f)$ is bounded below by collision frequency $\nu(v,I)$ for $t\in[\tilde{t},T_0]$, where $\tilde{t}$ will be defined in Lemma \ref{Rf est}.
\subsection{$\mathcal{R}(f)$ estimate}
In the following lemma, we use the equivalent form \eqref{BLCO} of the collision operator to obtain the lower bound for $\mathcal{R}(f)$ instead of \eqref{CO}. 
\begin{lemma} \label{Rf est}
	Assume that  
	\begin{align}\label{HC}
		\sup_{0\leq s\leq t}\|h(t)\|_{L^\infty} \leq \bar{M}.
	\end{align}
	Then, there exists a positive constant $C_{\delta,2}>0$ such that for any $T_0>\tilde{t}$ where
	\begin{align*}
		\tilde{t}:=\frac{1}{\nu_0}\ln\left(4C_{\beta,\nu_0} C_{\delta,2} M_0\right)>0,
	\end{align*}
	there also exists a small positive constant $\varepsilon=\varepsilon(\bar{M},T_0)>0$. Whenever $\mathcal{E}(F_0) \leq \varepsilon$, it holds that  
	\begin{align}
		\mathcal{R}(f)(t,x,v,I)\geq \frac{1}{2}\nu(v,I), \quad \forall (t,x,v,I) \in [\tilde{t},T_0]\times \T^3 \times \R^3 \times \R_+, 
	\end{align}
	where $\nu(v,I)$ is the collision frequency defined in \eqref{collision frequency}.
\end{lemma}
\begin{proof}
	Let us recall the definition of $\mathcal{R}(f)$:
		\begin{align*}
		\mathcal{R}(f)=\nu(v,I)+\int_{(\R^3\times\R_+)^3}W(v,v_{*},I,I_{*} \vert v',v'_{*},I',I'_{*})\sqrt{M_*}f_*dv_{*}dv'dv'_{*}dI_{*}dI'dI'_{*}.
	\end{align*}  
	Using the equivalent form \eqref{BLCO} of the collision operator \eqref{CO}, we can express $\mathcal{R}(f)$ by
	\begin{align}\label{RF2}
		\mathcal{R}(f)=\nu(v,I)+\int_{\S^2 \times [0,1]^2 \times \R^3 \times \R_+}B\sqrt{M_*}f_*(r(1-r))^{\delta/2-1}(1-R)^{\delta-1}R^{1/2}d\omega drdRdv_*dI_*.
	\end{align}
	Through directly computing integrals with respect to $r$ and $R$, we have
	\begin{align*}
		\mathcal{R}(f)=\nu(v,I)+C_\delta\int_{\R^3 \times \R_+}B\sqrt{M_*}f_*dv_*dI_*,
	\end{align*}
	where the constant $C_\delta$ depends only on $\delta$. We also use following property
	\begin{align*}
		\left(\frac{|v-v_*|^2}{4}+I+I_*\right)^{1-\alpha/2}&\leq(1+|v|^2+I)^{1-\alpha/2}(1+|v_*|^2+I_*)^{1-\alpha/2}\cr
		&\leq(1+|v|+\sqrt{I})^{2-\alpha}(1+|v_*|+\sqrt{I_*})^{2-\alpha}\cr
		&\leq C\nu(v,I)\nu(v_*,I_*),
	\end{align*}
	where the last inequality was obtained by $\nu(v,I) \simeq (1+|v|+\sqrt{I})^{2-\alpha}$. Then, $\mathcal{R}(f)$ has lower bound as follows:
	\begin{align}\label{OP}
		\begin{split}
			\mathcal{R}(f)&\geq \nu(v,I)\left\{1-C_{\delta,1}\int_{\R^3 \times \R_+}\nu(v_*,I_*)\sqrt{M_*}|f(t,x,v_*,I_*)|dv_*dI_*\right\}\cr
			&\geq \nu(v,I)\left\{1-C_{\delta,2}\int_{\R^3 \times 	\R_+}e^{-\frac{|v_*|^2}{4}-\frac{I_*}{4}}|h(t,x,v_*,I_*)|dv_*dI_*\right\},
		\end{split}
	\end{align}
	where $h(t,x,v_*,I_*)=w(v_*,I_*)f(t,x,v_*,I_*)$.
	Therefore, our goal is to find the constant $C_{\delta,2}>0$ that satisfies the following condition:
\begin{align}\label{OP2}
	\int_{\R^3 \times \R_+}e^{-\frac{|v|^2}{4}-\frac{I}{4}}|h(t,x,v,I)|dvdI\leq \frac{1}{2C_{\delta,2}}.
\end{align}
First, applying Duhamel's principle to \eqref{re BE}, we express $h(t,x,v,I)$ as shown below:
\begin{align*}
	h(t,x,v,I)=e^{-\nu(v,I)t}h_0(x-tv,v,I)+\int_{0}^{t}e^{-\nu(v,I)(t-s)}\left[K_wh+w(v,I)\Gamma(f,f)\right](s,x-(t-s)v,v,I)ds.
\end{align*}
Then, from the above formulation, the left-hand side of \eqref{OP2} can be expressed by
	\begin{align}\label{R OP}
		\begin{split}
			&\int_{\R^3 \times \R_+}e^{-\frac{|v |^2}{4}-\frac{I }{4}}|h(t,x,v,I)|dv dI \cr
			&\leq\int_{\R^3 \times \R_+}e^{-\nu(v ,I )t}e^{-\frac{|v |^2}{4}-\frac{I }{4}}|h_0(x-tv,v ,I )|dv dI \cr
			&+\int_{t-\lambda_{1}}^{t}\int_{\R^3 \times \R_+}e^{-\nu(v ,I )(t-s)}e^{-\frac{|v |^2}{4}-\frac{I }{4}}\left| [K_wh+w(v ,I )\Gamma(f ,f )] (s,x-(t-s)v ,v ,I )\right|dv dI ds\cr
			&+\int_{0}^{t-\lambda_{1}}\int_{\R^3 \times \R_+}e^{-\nu(v ,I )(t-s)}e^{-\frac{|v |^2}{4}-\frac{I }{4}}\left|[K_wh+w(v ,I )\Gamma(f ,f )](s,x-(t-s)v ,v ,I )\right|dv dI ds,
		\end{split}
	\end{align}
	where $\lambda_{1}>0$ is a small positive constant to be chosen later. From now on, we estimate the right-hand side of \eqref{R OP}. For the initial term, we obtain though direct computation:
	\begin{align}\label{IT}
		\int_{\R^3 \times \R_+}e^{-\nu(v,I)t}e^{-\frac{|v|^2}{4}-\frac{I}{4}}|h_0(x-tv,v,I)|dvdI\leq C e^{-\nu_0 t}\|h_0\|_{L^\infty}. 
	\end{align} 
	\hide where the constant $C_\beta>0$ satisfies
	\begin{align*}
			\int_{\R^3 \times \R_+}\frac{1}{(1+|v|+\sqrt{I})^{2\beta}}dvdI&\leq \int_{\R^3 \times \R_+}\frac{1}{(1+|v|)^{\beta}(1+\sqrt{I})^{\beta}}dvdI\cr
			&\leq\int_{\R^3}\frac{1}{(1+|v|)^{\beta}}dv\int_{\R_+}\frac{1}{(1+\sqrt{I})^{\beta}}dI\cr
			&\leq C_\beta,
	\end{align*}
	for $\beta>7$. 
	\unhide
	For other terms, using the lower bound of collision frequency, Lemma \ref{K esti}, and Lemma \ref{nonlinear esti}, we can further control the following terms by 
	\begin{align}\label{OT}
		e^{-\nu(v ,I )(t-s)}\left|[K_wh+w(v ,I )\Gamma(f ,f )](s,x-(t-s)v,v ,I) \right| \leq  C_{\beta}e^{-\frac{\nu_0}{2}(t-s)}\left\{\|h(s)\|_{L^\infty}+\|h(s)\|^2_{L^\infty}\right\}.
	\end{align}
	Therefore, integrating \eqref{OT} over $[t-\lambda_{1},t]$ gives
	\begin{align*}
		&\int_{t-\lambda_{1}}^{t}\int_{\R^3 \times \R_+}e^{-\nu(v ,I )(t-s)}e^{-\frac{|v |^2}{4}-\frac{I }{4}}\left|[K_wh+w(v ,I )\Gamma(f ,f )](s,x-(t-s)v ,v ,I )\right|dv dI ds \cr
		&\leq C_{\beta}\int_{t-\lambda_{1}}^{t}e^{-\frac{\nu_0}{2}(t-s)}\left\{\|h(s)\|_{L^\infty}+\|h(s)\|^2_{L^\infty}\right\}ds \cr
		&\leq C_{\beta}\lambda_{1}\sup_{0\leq s\leq t}\left\{\|h(s)\|_{L^\infty}+\|h(s)\|^2_{L^\infty}\right\}.
	\end{align*}
	Henceforth, we consider the remaining part that is integration over $[0,t-\lambda_{1}]$. For this case, we split into four parts as follows:
	\begin{align*}
		(1)&\ |v|\geq N \quad \text{or}\quad I\geq N, \\
		(2)&\  |v|\leq N,\ |v_*|\geq 2N, \ I \leq N, \cr
		(3)&\ |v| \leq N, \ |v_*|\leq 2N,\ I \leq N, \ I_*\geq 2N,\cr
		(4)&\ |v|\leq N,\ |v_*|\leq 2N,\ I\leq N, \ I_*\leq 2N.
	\end{align*}
	For  case $(1)$, using Lemma \ref{K esti}, it holds that
	\begin{align}\label{C1}
		\begin{split}
				&\int_{0}^{t-\lambda_{1}}\int_{\R^3 \times \R_+}e^{-\nu(v,I)(t-s)}e^{-\frac{|v|^2}{4}-\frac{I}{4}}\left|K_wh(s,x-(t-s)v,v,I)\right|\left(\chi_{\{|v|\geq N\}}+\chi_{\{I\geq N\}}\right) dvdIds \cr
			&\leq
			\frac{C_{\beta}}{N}\int_{0}^{t-\lambda_{1}}e^{-\frac{\nu_0}{2}(t-s)}\|h(s)\|_{L^\infty}ds \cr
			&\leq \frac{C_{\beta,\nu_0}}{N} \sup_{0\leq s\leq t}\|h(s)\|_{L^\infty}.
		\end{split}
	\end{align}
	For case $(2)$, it holds that $|v-v_*|>N$. Then,
	\begin{align}\label{C2}
		\begin{split}
				&\int_{0}^{t-\lambda_{1}}\int_{|v|\leq N, |v_*|\geq 2N, I\leq N, I_* \in \R_+ }e^{-\nu(v,I)(t-s)}e^{-\frac{|v|^2}{4}-\frac{I}{4}}|k_wh(s,x-(t-s)v,v_*,I_*)|dv_*dI_*dvdIds\cr
			&\leq C_{\nu_0}\sup_{0\leq s\leq t}\|h(s)\|_{L^\infty}\int_{|v|\leq N, |v_*|\geq 2N, I\leq N, I_* \in \R_+ }e^{-\frac{|v|^2}{4}-\frac{I}{4}}|k_w(v,v_*,I,I_*)|dv_*dI_*dvdI\cr
			&\leq  C_{\nu_0}\sup_{0\leq s\leq t}\|h(s)\|_{L^\infty}\int_{|v|\leq N, |v_*|\geq 2N, I\leq N, I_* \in \R_+ }e^{-\frac{|v|^2}{4}-\frac{I}{4}}\frac{e^{\frac{|v-v_*|^2}{64}}}{e^{\frac{N^2}{64}}}|k_w(v,v_*,I,I_*)|dv_*dI_*dvdI\cr
			&\leq\frac{C_{\nu_0}}{N}\sup_{0\leq s\leq t}\|h(s)\|_{L^\infty}\int_{\R^3\times \R_+}e^{-\frac{|v|^2}{4}-\frac{I}{4}}\int_{\R^3\times \R_+}e^{\frac{|v-v_*|^2}{64}}| k_w(v,v_*,I,I_*) | dv_*dI_*dvdI\cr
			&\leq\frac{C_{\nu_0}}{N}\sup_{0\leq s\leq t}\|h(s)\|_{L^\infty},
		\end{split}
	\end{align}
	where the last inequality comes from Lemma \ref{K esti}. For case $(3)$, it is directly estimated by 
	\begin{align}\label{C3}
		\begin{split}
		&\int_{0}^{t-\lambda_{1}}\int_{|v|\leq N, |v_*|\leq 2N, I\leq N, I_* \geq 2N }e^{-\nu(v,I)(t-s)}e^{-\frac{|v|^2}{4}-\frac{I}{4}}|k_wh(s,x-(t-s)v,v_*,I_*)|dv_*dI_*dvdIds\cr
		&\leq\sup_{0\leq s\leq t}\|h(s)\|_{L^\infty}\int_{0}^{t-\lambda_{1}}e^{- \frac{\nu_0}{2}(t-s)}\int_{|v|\leq N, |v_*|\leq 2N, I\leq N, I_* \geq 2N }\left(\frac{1+I_*}{1+I_*}\right)^{1/8}|k_w(v,v_*,I,I_*)|dv_*dI_*dvdIds\cr
		&\leq\frac{C_{\nu_0}}{N^{1/8}}\sup_{0\leq s\leq t}\|h(s)\|_{L^\infty}.
		\end{split}
	\end{align}
	For case $(4)$, using H\"older's inequality and Lemma \ref{K esti}, we have
		\begin{align*}
			&\int_{0}^{t-\lambda_{1}}\int_{|v|\leq N, |v_*|\leq 2N,I\leq N,I_*\leq 2N  }e^{- \frac{\nu_0}{2}(t-s)}e^{-\frac{|v|^2}{4}-\frac{I}{4}}\left | k_wh(s,x-(t-s)v,v_*,I_*)\right | dv_*dI_*dvdIds\cr
			&\leq \int_{0}^{t-\lambda_{1}}e^{- \frac{\nu_0}{2}(t-s)}\left(\int_{|v|\leq N, |v_*|\leq 2N,I\leq N,I_*\leq 2N  }e^{-\frac{|v|^2}{2}-\frac{I}{2}}|k_w(v,v_*,I,I_*)|^2dv_*dI_*dvdI\right)^{1/2}\cr
			&\quad \times\left(\int_{|v|\leq N,|v_*|\leq 2N,I\leq N,I_*\leq 2N  }e^{-\frac{|v|^2}{2}-\frac{I}{2}}|h(s,x-(t-s)v,v_*,I_*)|^2dv_*dI_*dvdI\right)^{1/2}ds\cr
			&\leq C\int_{0}^{t-\lambda_{1}}e^{- \frac{\nu_0}{2}(t-s)} \left(\int_{|v|\leq N, |v_*|\leq 2N,I\leq N,I_*\leq 2N  }e^{-\frac{|v|^2}{2}-\frac{I}{2}}|h(s,x-(t-s)v,v_*,I_*)|^2dv_*dI_*dvdI\right)^{1/2}ds. 
		\end{align*}
To estimate the above term, we use the change of variable $y:=x-(t-s)v$ to obtain the following estimate 
		\begin{align}\label{Kn}
			\begin{split}
				&\int_{|v|\leq N, |v_*|\leq 2N,I\leq N,I_*\leq 2N  }e^{-\frac{|v|^2}{2}-\frac{I}{2}}|h(s,x-(t-s)v,v_*,I_*)|^2dv_*dI_*dvdI\cr
				&\leq C_{\beta}\frac{1+N(t-s)^3}{(t-s)^3}\int_{\mathbb{T}^3}\int_{|v_*|\leq2N,I\leq N,I_*\leq2N} e^{-\frac{|v|^2}{2}-\frac{I}{2}}w^2(v_*,I_*)|f(s,y,v_*,I_*)|^2dv_*dI_*dIdy\cr
				&\leq C_{\beta}\left(\frac{1}{\lambda_{1}^3}+N\right)\int_{\mathbb{T}^3}\int_{|v_*|\leq2N,I\leq N,I_*\leq2N}| f(s,y,v_*,I_*)|^2dv_*dI_*dIdy.
			\end{split}
		\end{align}
		To finish the estimate for this term, we split it into two parts as:
		\begin{align}\label{F2}
			\begin{split}
				&\int_{\mathbb{T}^3}\int_{|v_*|\leq2N,I_*\leq2N}|f(s,y,v_*,I_*)|^2dv_*dI_*dy\cr
				& =\int_{\mathbb{T}^3}\int_{|v_*|\leq2N,I_*\leq\lambda_{2}}|f(s,y,v_*,I_*)|^2dv_*dI_*dy+ \int_{\mathbb{T}^3}\int_{|v_*|\leq2N,\lambda_{2}\leq I_*\leq2N}|f(s,y,v_*,I_*)|^2dv_*dI_*dy\cr
				&:=I_1+I_2
			\end{split}
	\end{align}
	For $I_1$, this is computed easily as follows: 
	\begin{align}\label{I1}
		I_1&\leq\lambda_{2}\|h(s)\|^2_{L^\infty}\int_{\mathbb{T}^3}\int_{|v_*|\leq2N}\frac{1}{(1+|v_*|)^{2\beta}}dv_*dy\leq \lambda_{2} C_{\beta}\|h(s)\|^2_{L^\infty}.
	\end{align}
	Applying Lemma \ref{relative entropy lemma} to $I_2$, one obtains that 
	\begin{align}\label{II2}
		\begin{split}
			I_2&\leq \int_{\mathbb{T}^3}\int_{|v_*|\leq2N,\lambda_{2}\leq I_*\leq2N}|f(s,y,v_*,I_*)|^2dv_*dI_*dy\cr
			&\quad+\int_{\mathbb{T}^3}\int_{|v_*|\leq2N,\lambda_{2}\leq 	I_*\leq2N}|f(s,y,v_*,I_*)|^2\left(\chi_{\{ \vert \sqrt{M_*}f_* \vert \leq M_*\}} + \chi_{\{\vert \sqrt{M_*}f_* \vert >M_*\}}\right)dv_*dI_*dy \cr
			&\leq  \mathcal{E}(F_0)+ \sup_{0\leq s \leq t} \Vert h(s) \Vert_{L^\infty} 	\int_{\mathbb{T}^3}\int_{|v_*|\leq2N,\lambda_{2}\leq I_*\leq2N} \frac{\sqrt{M_*}}{\sqrt{M_*}}\vert f(s', y, v_*,I_*)\vert \chi_{\{\vert \sqrt{M_*}f_* \vert > M_*\}}dv_*dI_*dy\cr
			&\leq\mathcal{E}(F_0)+C\frac{e^{2N}}{\lambda_{2}^{(\delta-2)/4}}  \sup_{0\leq s \leq t} \Vert h(s) 	\Vert_{L^\infty}  \mathcal{E}(F_0).
		\end{split}
	\end{align}
	Thus, \eqref{Kn} can be futrher bounded by \eqref{I1} and \eqref{II2}
		\begin{align*}
		&\int_{|v|\leq N,|v_*|\leq 2N,I\leq N,I_*\leq 2N  }e^{-\frac{|v|^2}{4}-\frac{I}{4}}|h(s,x-(t-s)v,v_*,I_*)|^2dvdIdv_*dI_*ds\cr
		&\leq C_{\beta}\left(\frac{1}{\lambda_{1}^3}+N\right)\left(\lambda_{2}\sup_{0\leq s\leq t}\|h(s)\|^2_{L^\infty}+\mathcal{E}(F_0)+\frac{e^{2N}}{\lambda_{2}^{(\delta-2)/4}}  \sup_{0\leq s \leq t} \Vert h(s) \Vert_{L^\infty}  \mathcal{E}(F_0)\right).
	\end{align*}
	In conclusion, we have
		\begin{align}\label{C4}
			\begin{split}
					&\int_{0}^{t-\lambda_{1}}\int_{|v|\leq N, |v_*|\leq 2N,I\leq N,I_*\leq 2N  }e^{- \frac{\nu_0}{2}(t-s)}e^{-\frac{|v|^2}{4}-\frac{I}{4}}\left | k_wh(s,x-(t-s)v,v_*,I_*)\right | dv_*dI_*dvdIds\cr &\leq C_{\beta,\nu_0}\left(\frac{1}{\lambda_{1}^3}+N\right)^{1/2}\left(\lambda_{2}\sup_{0\leq s\leq t}\|h(s)\|^2_{L^\infty}+\mathcal{E}(F_0)+\frac{e^{2N}}{\lambda_{2}^{(\delta-2)/4}}  \sup_{0\leq s \leq t} \Vert h(s) \Vert_{L^\infty}  \mathcal{E}(F_0)\right)^{1/2}.
			\end{split}
	\end{align}
	Combining \eqref{C1}, \eqref{C2}, \eqref{C3} and \eqref{C4}, then we have
		\begin{align}\label{Rkw}
			\begin{split}
				&\int_{0}^{t-\lambda_{1}}\int_{\R^3\times\R_+ }e^{-\nu(v,I)(t-s)}e^{-\frac{|v|^2}{4}-\frac{I}{4}}\left|K_wh(s,x-(t-s)v,v,I)\right|dvdIds \cr
				&\leq \frac{C_{\beta,\nu_0}}{N} \sup_{0\leq s\leq t}\|h(s)\|_{L^\infty}+\frac{C_{\nu_0}}{N}\sup_{0\leq s\leq t}\|h(s)\|_{L^\infty}+\frac{C_{\nu_0}}{N^{1/8}}\sup_{0\leq s\leq t}\|h(s)\|_{L^\infty}\cr
				&\quad+C_{\beta\nu_0}\left(\frac{1}{\lambda_{1}^3}+N\right)^{1/2}\left(\lambda_{2}\sup_{0\leq s\leq t}\|h(s)\|^2_{L^\infty}+\mathcal{E}(F_0)+\frac{e^{2N}}{\lambda_{2}^{(\delta-2)/4}}  \sup_{0\leq s \leq t} \Vert h(s) \Vert_{L^\infty}  \mathcal{E}(F_0)\right)^{1/2}.
			\end{split}
	\end{align}
	From now on, we consider the remaining term, which is the non-linear part. Using Lemma \ref{GE}, we can estimate the part $w\Gamma_+(f,f)$ as follows:
	\begin{align}\label{OI}
		\begin{split}
			&\int_{0}^{t-\lambda_{1}}e^{-\nu(v,I)(t-s)}\int_{|v|\leq N,I\leq N}e^{-\frac{|v|^2}{4}-\frac{I}{4}} |w(v,I)\Gamma_+(f,f)(s,x-(t-s)v,v)|dvdIds\cr
			&\leq C\int_{0}^{t-\lambda_{1}}e^{- \frac{\nu_0}{2}(t-s)}\int_{|v|\leq N,I\leq N}e^{-\frac{|v|^2}{4}-\frac{I}{4}}\\
			&\quad \times \frac{\sup_{0\leq s \leq t}\|h(s)\|_{L^\infty}}{1+|v|+I^{1/4}} \left(\int_{\R^3 \times \R_+} (1+\vert \eta \vert + \sqrt{L})^{-2\beta+8} \vert wf(\eta,L) \vert^2 d\eta dL \right)^{1/2}dvdIds\cr
			&\leq C\sup_{0\leq s \leq t}\|h(s)\|_{L^\infty}\int_{0}^{t-\lambda_{1}}e^{- \frac{\nu_0}{2}(t-s)}\int_{|v|\leq N,I\leq N}e^{-\frac{|v|^2}{4}-\frac{I}{4}}\\
			&\quad \times \left(\int_{\R^3 \times \R_+} (1+\vert \eta \vert + \sqrt{L})^{-2\beta+8} \vert wf(\eta,L) \vert^2 d\eta dL \right)^{1/2}dvdIds.
		\end{split}
	\end{align}
	We consider following two parts:
\begin{align}\label{GE1}
			\begin{split}
				&\int_{|v|\leq N,I\leq N}e^{-\frac{|v|^2}{4}-\frac{I}{4}}\left(\int_{\R^3 \times \R_+} (1+\vert \eta \vert + \sqrt{L})^{-2\beta+8} \vert h(\eta,L) \vert^2 d\eta dL \right)^{1/2}dvdI\cr
				&\leq C\int_{|v|\leq N,I\leq N}e^{-\frac{|v|^2}{4}-\frac{I}{4}}\\
				&\quad \times \left\{\left(\int_{\R^3 \times \R_+}(\chi_{\{|\eta| >N\}}+\chi_{\{L>N\}})+\int_{|\eta|\leq N,L\leq N} \right)\frac{1}{(1+\vert \eta \vert + \sqrt{L})^{2\beta-8}} \vert h(\eta,L) \vert^2 d\eta dL\right\}^{1/2}dvdI.
			\end{split}
	\end{align}
	For $\beta>7$, the first term in \eqref{GE1} can be bounded by
	\begin{align}\label{GE2}
			\begin{split}
				&\int_{|v|\leq N,I\leq N}e^{-\frac{|v|^2}{4}-\frac{I}{4}}\int_{\R^3 \times \R_+}\frac{1}{(1+\vert \eta \vert + \sqrt{L})^{2\beta-8}} \vert h(\eta,L) \vert^2(\chi_{\{|\eta| >N\}}+\chi_{\{L>N\}}) d\eta dLdvdI \cr
				& \leq \sup_{0\leq s \leq t}\|h(s)\|^2_{L^\infty}\int_{\R^3 \times \R_+}\frac{1}{(1+\vert \eta \vert + \sqrt{L})^{2\beta-9}}\frac{1}{(1+\vert \eta \vert + \sqrt{L})}(\chi_{\{|\eta| >N\}}+\chi_{\{L>N\}})d\eta dL\cr
				&\leq\frac{\sup_{0\leq s \leq t}\|h(s)\|^2_{L^\infty}}{N^{1/2}}\int_{\R^3 \times \R_+}\frac{1}{(1+\vert \eta \vert)^{3+(\beta -7)}(1+\vert \eta \vert + \sqrt{L})^{\beta-5}} (\chi_{\{|\eta| >N\}}+\chi_{\{L>N\}})d\eta dL\cr
				& \leq \frac{C_\beta\sup_{0\leq s \leq t}\|h(s)\|^2_{L^\infty}}{N^{1/2}}.
			\end{split}
	\end{align}
	We use H\"older's inequality to estimate the remaining term in \eqref{GE1}, and then
\begin{align}\label{GEH}
	\begin{split}	
		&\int_{|v|\leq N,I\leq N}e^{-\frac{|v|^2}{4}-\frac{I}{4}}\left(\int_{|v|\leq N,I\leq N} (1+\vert \eta \vert + \sqrt{L})^{-2\beta+8} \vert h(\eta,L) \vert^2 d\eta dL \right)^{1/2}dvdI\cr
		&\quad\leq C\left(\int_{|v|\leq N,I\leq N}e^{-\frac{|v|^2}{2}-\frac{I}{2}}dvdI\right)^2\left\{\int_{|v|\leq N,I\leq N,|\eta|\leq N,L\leq N}\frac{1}{(1+\vert \eta \vert + \sqrt{L})^{2\beta-8}} \vert h(\eta,L) \vert^2 d\eta dLdvdI\right\}^{1/2}\cr
		&\quad\leq \left\{\int_{|v|\leq N,I\leq N,|\eta|\leq N,L\leq N}\frac{1}{(1+\vert \eta \vert + \sqrt{L})^{2\beta-8}} \vert h(\eta,L) \vert^2 d\eta dLdvdI\right\}^{1/2}.
	\end{split}
	\end{align}
	For \eqref{GEH}, we compute the following integrals 
	\begin{align*}
		&\int_{|v|\leq N,I\leq N,|\eta|\leq N,I\leq N}\frac{w^2(\eta,L)}{(1+\vert \eta \vert + \sqrt{L})^{2\beta-8}} \vert f(\eta,L) \vert^2 d\eta dLdvdI \cr
		&\leq (1+2N)^8\int_{|v|\leq N,I\leq N,|\eta|\leq N,L\leq N}\vert f(s,x-(t-s)v,\eta,L) \vert^2 d\eta dLdvdI.
	\end{align*}
	Using the change of variable $y:=x-(t-s)v$, we have
	\begin{align*}
		&\int_{|v|\leq N,I\leq N,|\eta|\leq N,L\leq N}\vert f(s,x-(t-s)v,\eta,L) \vert^2 d\eta dLdvdI \cr
		&\leq C\frac{1+N(t-s)^3}{(t-s)^3}\int_{\mathbb{T}^3}\int_{I\leq N, |\eta|\leq N,L\leq N}\vert f(s,y,\eta,L) \vert^2 d\eta dLdIdy.
	\end{align*}
	From the estimate obtained in \eqref{II2}, we get
	\begin{align}\label{GE3}
		\begin{split}
			&\int_{|v|\leq N,I\leq N,|\eta|\leq N,Lc\leq N}\frac{w^2(\eta,L)}{(1+\vert \eta \vert + \sqrt{L})^{2\beta-8}} \vert f(\eta,L) \vert^2 d\eta dLdvdI \cr
			&\leq (1+2N)^8\frac{1+N(t-s)^3}{(t-s)^3}\left\{\lambda_{2}\sup_{0\leq s\leq t}\|h(s)\|^2_{L^\infty}+\mathcal{E}(F_0)+\frac{e^{2N}}{\lambda_{2}^{(\delta-2)/4}}  \sup_{0\leq s \leq t} \Vert h(s) 	\Vert_{L^\infty}  \mathcal{E}(F_0)\right\}.
		\end{split}
	\end{align}
	After combining \eqref{GE2} and \eqref{GE3}, and inserting them into \eqref{OI}, we have
	\begin{align}\label{RG+}
		\begin{split}
				&\sup_{0\leq s \leq t}\|h(s)\|_{L^\infty}\int_{0}^{t-\lambda_{1}}e^{- \frac{\nu_0}{2}(t-s)}\int_{|v|\leq N,I\leq N}e^{-\frac{|v|^2}{4}-\frac{I}{4}}\\
			&\quad \times \left(\int_{\R^3 \times \R_+} (1+\vert \eta \vert + \sqrt{L})^{-2\beta+8} \vert wf(\eta,L) \vert^2 d\eta dL \right)^{1/2}dvdIds\cr
			&\leq C_{\beta,\nu_0}(1+2N)^4\left(\frac{1}{\lambda_{1}^3}+N\right)^{1/2} \sup_{0\leq s \leq t}\|h(s)\|_{L^\infty}\\
			&\quad \times \left\{\frac{C_\beta}{N^{1/4}}\sup_{0\leq s \leq t}\|h(s)\|^2_{L^\infty}+\lambda_{2}\sup_{0\leq s\leq t}\|h(s)\|^2_{L^\infty}+\mathcal{E}(F_0)+\frac{e^{2N}}{\lambda_{2}^{(\delta-2)/4}}  \sup_{0\leq s \leq t} \Vert h(s) 	\Vert_{L^\infty}  \mathcal{E}(F_0)\right\}^{1/2}.
		\end{split}
	\end{align} 
	To estimate the part $w\Gamma_-(f,f)$, we consider the form \eqref{BLCO} of the collision operator. Since the collision frequency $\nu(v,I)$ is equivalent to $(1+|v|+\sqrt{I})^{2-\alpha}$, we have
	\begin{align*}
		|w(v,I)\Gamma_-(f,f)(s,x-(t-s)v,v,I)|&\leq\|h(s)\|_{L^\infty}\int_{\R^3\times\R_+}\left(\frac{|v-v_*|^2}{2}+I+I_*\right)^{1-\alpha/2}\sqrt{M_*}|f_*|dv_*dI_*\cr
		&\leq C\|h(s)\|_{L^\infty}\nu(v,I)\int_{\R^3\times\R_+}\nu(v_*,I_*)\sqrt{M_*}|f_*|dv_*dI_*.
	\end{align*}
	It follows from H\"older's inequality that
	\begin{align*}
		&\int_{\R^3\times\R_+}\nu(v_*,I_*)\sqrt{M_*}|f_*|dv_*dI_*\\
		&\leq\left(\int_{\R^3 \times \R_+}\nu^2(v_*,I_*)M_*dv_*dI_*\right)^{1/2}\left(\int_{\R^3 \times \R_+}|f(s,x-(t-s)v,v_*,I_*)|^2dv_*dI_*\right)^{1/2}\cr
		&\leq C\left(\int_{\R^3 \times \R_+}|f(s,x-(t-s)v,v_*,I_*)|^2dv_*dI_*\right)^{1/2}.
	\end{align*}
	Thus, by using a similar approach to dealing with the part $w\Gamma_+(f,f)$, we can treat  $w\Gamma_-(f,f)$ as follows:
	\begin{align}\label{RG-}
		\begin{split}
				&\int_{0}^{t-\lambda_{1}}e^{-\nu(v,I)(t-s)}\int_{|v|\leq N,I\leq N}e^{-\frac{|v|^2}{4}-\frac{I}{4}}|w(v,I)\Gamma_-(f,f)(s,x-(t-s)v,v,I)|dvdIds\cr
		&\leq C\sup_{0\leq s \leq t}\|h(s)\|_{L^\infty}\int_{0}^{t-\lambda_{1}}e^{- \frac{\nu_0}{2}(t-s)}\int_{|v|\leq N,I\leq N}e^{-\frac{|v|^2}{4}-\frac{I}{4}}\nu(v,I)\cr
		&\quad \times\left(\int_{\R^3 \times \R_+}|f(s,x-(t-s)v,v_*,I_*)|^2dv_*dI_*\right)^{1/2}dvdIds\cr
		&\leq C_{\beta,\nu_0}\left(\frac{1}{\lambda_{1}^3}+N\right)\sup_{0\leq s\leq t}\|h(s)\|_{L^\infty}\\
		&\quad \times \left\{\frac{C_\beta}{N^{1/4}}\sup_{0\leq s \leq t}\|h(s)\|^2_{L^\infty}+\lambda_{2}\sup_{0\leq s\leq t}\|h(s)\|^2_{L^\infty}+\mathcal{E}(F_0)+\frac{e^{2N}}{\lambda_{2}^{(\delta-2)/4}}  \sup_{0\leq s \leq t} \Vert h(s) 	\Vert_{L^\infty}  \mathcal{E}(F_0)\right\}^{1/2}.
		\end{split}
	\end{align}
	In conclusion, combining \eqref{IT}, \eqref{OT}, \eqref{Rkw} \eqref{RG+} and \eqref{RG-},  and using \eqref{HC}, we get
	\begin{align} \label{Rf esti final} 
		\begin{split}
		&\int_{\R^3\times\R_+}e^{-\frac{|v|^2}{4}-\frac{I}{4}}|h(t,x,v,I)|dvdI  \\ 
		&\leq C_{\beta,\nu_0} e^{-\nu_0 t}\|h_0\|_{L^\infty}  \\
		&\quad + C_{\beta,\nu_0}\lambda_{1}\left\{\bar{M}+\bar{M}^2\right\} +\frac{C_{\beta,\nu_0}}{N}\bar{M}+ \frac{C_{\nu_0}}{N}\bar{M}+\frac{C_{\beta,\nu_0}}{N^{1/8}}\bar{M}\cr
		&\quad+ C_{\beta,\nu_0}\left(\frac{1}{\lambda_{1}^3}+N\right)^{1/2}\left\{\lambda_{2}\bar{M}^2+\mathcal{E}(F_0)+\frac{e^{2N}}{\lambda_{2}^{(\delta-2)/4}}  \bar{M} \mathcal{E}(F_0)\right\}^{1/2}\cr
		&\quad+C_{\beta,\nu_0}(1+2N)^4\left(\frac{1}{\lambda_{1}^3}+N\right)^{1/2}\bar{M}\left\{\frac{C_\beta}{N^{1/4}}\bar{M}+\lambda_{2}\bar{M}^2+\mathcal{E}(F_0)+\frac{e^{2N}}{\lambda_{2}^{(\delta-2)/4}}  \bar{M} \mathcal{E}(F_0)\right\}^{1/2}\cr
		&\quad +C_{\beta,\nu_0}\left(\frac{1}{\lambda_{1}^3}+N\right)\bar{M}\left\{\frac{C_\beta}{N^{1/4}}\bar{M}^2+\lambda_{2}\bar{M}^2+\mathcal{E}(F_0)+\frac{e^{2N}}{\lambda_{2}^{(\delta-2)/4}}  \bar{M} \mathcal{E}(F_0)\right\}^{1/2}\cr
		&:= J_1 + J_2.
		\end{split}
	\end{align}
	Here, we choose the time $\tilde{t}=\frac{1}{\nu_0}\ln\left(4C_{\beta,\nu_0} C_{\delta,2} M_0\right)$ so that
	\begin{align*}
			J_1 \leq \frac{1}{4C_{\delta,2}}, \quad t \geq \tilde{t},
	\end{align*}
		where 
	\begin{align*}
			J_1:=C_{\beta,\nu_0} e^{-\nu_0 t}\|h_0\|_{L^\infty}.
	\end{align*}
	To get the desired result, we firstly choose $\lambda_{1}>0$ sufficiently small depending on $\bar{M}$. Also, we choose $N$ large enough depending on $\lambda_1$ and $\bar{M}$, then we choose the constant $\lambda_{2}>0$ small enough depending on $\lambda_{1}$, $N$ and $\bar{M}$. Finally, we choose $\varepsilon$ sufficiently small satisfying $\mathcal{E}(F_0) \leq \varepsilon$ so that 
	\begin{align*}
		J_2  \leq \frac{1}{4C_{\delta,2}}.
	\end{align*}
	In conclusion, we get the following that
	\begin{align*}
		\int_{\R^3\times\R_+}e^{-\frac{|v|^2}{4}-\frac{I}{4}}|h(t,x,v,I)|dvdI \leq \frac{1}{2C_{\delta,2}}, \quad \forall (t,x) \in [\tilde{t},T_0]\times \T^3. 
	\end{align*}
	This proof is complete.
\end{proof}

\subsection{$L^\infty$ estimate}
To treat $L^{\infty}$ estimate over $[0,T_0]$, we apply Duhamel's principle to the reformulated Boltzmann equation 
\begin{align*}
	\p_t f + v\cdot \nabla_x f + \mathcal{R}(f) f =Kf +\Gamma_+ (f,f).
\end{align*} 
By multiplying the velocity weight function $w(v,I)$ onto the equation above, we can express $h(t,x,v,I) = w(v,I)f(t,x,v,I)$ from Duhamel's principle: 
\begin{align*}
	h(t,x,v,I) = e^{-\int_0^t \mathcal{R}(f)(s,x-(t-s)v,v) ds} h(0,x-tv,v) + \int_0^t e^{-\int_s^t \mathcal{R}(f)(\tau, x-(t-\tau)v,v)d\tau} \left[K_w h(s)+ w \Gamma_+(f,f)(s)\right] ds,
\end{align*}
where the operator $K_w$ is given by 
\begin{align} \label{kw}
	K_w f : = \int_{\R^3 \times \R_+} k_w(v,v_*,I,I_*) f(t,x,v_*,I_*) dv_* dI_*. 
\end{align}
\begin{lemma} \label{L infty est}
	Assume \eqref{HC} in Lemma \ref{Rf est} and $\mathcal{E}(F_0) \leq \varepsilon.$ Let $h=h(t,x,v,I)$ be the solution to the reformulated Boltzmann equation 
	\begin{align*}
		\p_t h + v \cdot \nabla_x h + \mathcal{R}(f) h = K_w h + w \Gamma_+(f,f).
	\end{align*}
	Then, for all $(t,x,v,I) \in [0,T_0]\times \T^3 \times \R^3 \times \R_+$, we have 
	\begin{align} \label{main estimate}
		\begin{split}
			\vert h(t,x,v,I) \vert &\leq \bar{C}e^{-\frac{\nu_0}{2}(t-\tilde{t})}\Vert h_0 \Vert_{L^\infty}\left(1+\int_{0}^{t}\|h(s)\|_{L^\infty}ds\right)\cr
			&\quad +\bar{C}e^{\frac{5\nu_0}{4}\tilde{t}}\left(\sup_{0\leq s \leq t}\|h(s)\|_{L^\infty}+\sup_{0\leq s \leq t}\|h(s)\|^2_{L^\infty}+\sup_{0\leq s \leq t}\|h(s)\|^3_{L^\infty}\right)\left[\lambda_{1}+\frac{1}{N}+C_{N,\lambda_{1}}\sqrt{\lambda_{2}}\right]\cr
			&\quad +\bar{C}e^{\nu_0 \tilde{t}} \sqrt{\mathcal{E}(F_0)}\\
			&\quad \quad \times \left[ \left\{C_{N,\lambda_1}\left(1+\sup_{0\leq s \leq t} \Vert h(s) \Vert_{L^{\infty}}\right) +C_{N,\lambda_1,\lambda_2} \left(\sup_{0\leq s \leq t} \Vert h(s) \Vert_{L^{\infty}}^{1/2}+\sup_{0\leq s \leq t} \Vert h(s) \Vert_{L^{\infty}}^{3/2}\right)\right\}\right]\cr
			&\quad +\bar{C}e^{\nu_0\tilde{t}}\sqrt{\mathcal{E}(F_0)}\\
			&\quad \quad \times \left(\sup_{0\leq s \leq t}\|h(s)\|_{L^\infty}+\sup_{0\leq s \leq t}\|h(s)\|^2_{L^\infty}\right)\left\{C_{N,\lambda_{1}}\left(1+C_{N,\lambda_{2}} \sup_{0\leq s \leq t} \Vert h(s) 	\Vert_{L^\infty}\right)\right\}^{\frac{1}{2}},
		\end{split}
	\end{align}  
	where $\lambda_1>0, \lambda_2>0$ and $\varepsilon>0$ can be arbitrarily small, and $N>0$ can be arbitrarily large. 
\end{lemma}
\begin{proof}
	Note that it follows from Lemma \ref{Rf est} that 
	\begin{align*}
		\mathcal{R}(f)(t,x,v,I) \geq 
		\begin{cases}
			0, \quad t \in [0,\tilde{t}], \\ 
			\frac{1}{2}\nu(v,I), \quad t \in [\tilde{t},\infty). 
		\end{cases}
	\end{align*}
	Thus, for all $t\in [0,T_0]$, one obtains that 
	\begin{align*}
		&e^{-\int_0^t \mathcal{R}(f)(\tau,x-(t-\tau)v,v,I) d\tau} \\
		&= e^{-\int_0^t \mathcal{R}(f)(\tau,x-(t-\tau)v,v,I) d\tau} \chi_{\{t \leq \tilde{t}\}}+ e^{-\int_0^{\tilde{t}} \mathcal{R}(f)(\tau,x-(t-\tau)v,v,I) d\tau}e^{-\int_{\tilde{t}}^t \mathcal{R}(f)(\tau,x-(t-\tau)v,v,I) d\tau} \chi_{\{t >\tilde{t}\}}\\
		&\leq e^{\frac{\nu_0}{2}\tilde{t}} e^{-\frac{\nu_0}{2}t} \chi_{\{t \leq \tilde{t}\}} + e^{\frac{\nu_0}{2}\tilde{t}} e^{-\frac{\nu_0}{2}t} \chi_{\{t > \tilde{t}\}}\\
		&\leq e^{\frac{\nu_0}{2}\tilde{t}} e^{-\frac{\nu_0}{2}t}. 
	\end{align*}
	Using the above fact, we have 
	\begin{align} \label{duhamel}
		\vert h(t,x,v,I) \vert &\leq e^{\frac{\nu_0}{2}\tilde{t}} e^{-\frac{\nu_0}{2}t} \Vert h_0 \Vert_{L^\infty} + \int_0^t e^{\frac{\nu_0}{2}\tilde{t}} e^{-\frac{\nu_0}{2}(t-s)} \vert K_w h (s) \vert  ds +  \int_0^t e^{\frac{\nu_0}{2}\tilde{t}} e^{-\frac{\nu_0}{2}(t-s)} \vert w \Gamma_+(f,f)(s) \vert ds. 	
	\end{align}
	Firstly, we treat the term containing $K_w h$ in \eqref{duhamel}. Using the definition \eqref{kw} of the operator $K_w$, we have 
	\begin{align*}
		\int_0^t e^{\frac{\nu_0}{2}\tilde{t}} e^{-\frac{\nu_0}{2}(t-s)} \vert K_w h (s) \vert  ds \leq  \int_0^t e^{\frac{\nu_0}{2}\tilde{t}} e^{-\frac{\nu_0}{2}(t-s)} \int_{\R^3 \times \R_+} \vert k_w (v,v_*, I,I_*) h(s,x-(t-s)v, v_*,I_*) \vert dv_* dI_* ds.
	\end{align*}
	We split it into several cases regarding $v_*$ and $I_*$.  \\ 
	
	\textbf{Case 1.} ($\vert v \vert \geq N$ or $I \geq N$) Using Lemma \ref{K esti}, we obtain
	\begin{align} \label{K case1}
		\begin{split}
			&\int_0^t e^{\frac{\nu_0}{2}\tilde{t}} e^{-\frac{\nu_0}{2}(t-s)} \int_{\R^3 \times \R_+} \vert k_w (v,v_*, I,I_*) h(s,x-(t-s)v, v_*,I_*) \vert (\chi_{\{\vert v \vert \geq N\}}+\chi_{\{I \geq N\}}) dv_* dI_* ds \\
			&\leq Ce^{\frac{\nu_0}{2} \tilde{t}}\left(\frac{1}{N}+ \frac{1}{N^{1/4}}\right) \sup_{0\leq s \leq t} \Vert h(s) \Vert_{L^{\infty}}. 
		\end{split}
	\end{align}
	In the subsequent cases, we only consider $\vert v \vert \leq N$ and $I \leq N$. \\
	
	\textbf{Case 2.} ($\vert v_* \vert \geq 2N$) In this case, due to $\vert v - v_* \vert \geq N$, we have
	\begin{align} \label{K case2}
		\begin{split}
			&\int_0^t e^{\frac{\nu_0}{2}\tilde{t}} e^{-\frac{\nu_0}{2}(t-s)} \int_{\R_+}\int_{\vert v_* \vert \geq 2N} \vert k_w (v,v_*, I,I_*) h(s,x-(t-s)v, v_*,I_*) \vert  dv_* dI_* ds \\
			&\leq  e^{-\frac{N^2}{32}} \int_0^t e^{\frac{\nu_0}{2}\tilde{t}} e^{-\frac{\nu_0}{2}(t-s)} \int_{\R_+}\int_{\vert v_* \vert \geq 2N} \vert k_w (v,v_*, I,I_*) h(s,x-(t-s)v, v_*,I_*) \vert e^{\frac{\vert v - v_* \vert ^2}{32}} dv_* dI_* ds \\
			& \leq C e^{-\frac{N^2}{32}} e^{\frac{\nu_0}{2} \tilde{t}} \sup_{0\leq s \leq t} \Vert h(s) \Vert_{L^{\infty}},		
		\end{split}
	\end{align}
	where we used Lemma \ref{K esti} in the last inequality.\\
	
	\textbf{Case 3.} ($\vert v_* \vert \leq 2N$ and $I_* \geq 2N$) Similar to Case 2, we use Lemma \ref{K esti} to get
	\begin{align} \label{K case3}
		\begin{split}
			&\int_0^t e^{\frac{\nu_0}{2}\tilde{t}} e^{-\frac{\nu_0}{2}(t-s)} \int_{I_* \geq 2N}\int_{\vert v_* \vert \leq 2N} \vert k_w (v,v_*, I,I_*) h(s,x-(t-s)v, v_*,I_*) \vert  dv_* dI_* ds\\
			&\leq \frac{1}{N^{1/8}} \int_0^t e^{\frac{\nu_0}{2}\tilde{t}} e^{-\frac{\nu_0}{2}(t-s)} \int_{I_* \geq 2N}\int_{\vert v_* \vert \leq 2N} \vert k_w (v,v_*, I,I_*) h(s,x-(t-s)v, v_*,I_*) \vert(1+I_*)^{1/8}  dv_* dI_* ds\\
			& \leq  \frac{1}{N^{1/8}} e^{\frac{\nu_0}{2} \tilde{t}} \sup_{0\leq s \leq t} \Vert h(s) \Vert_{L^{\infty}}. 
		\end{split}
	\end{align}
	Hence, we only consider the case $\{\vert v_* \vert \leq 2N, I_* \leq 2N\}$. \\
	
	\textbf{Case 4.}($\vert v_* \vert \leq 2N, I_* \leq 2N$)
	We will apply Duhamel's principle once again, and separate into the following three parts as shown below
	\begin{align} \label{kw case}
		\begin{split}
			&\int_0^t e^{\frac{\nu_0}{2}\tilde{t}} e^{-\frac{\nu_0}{2}(t-s)} \int_{I_* \leq 2N}\int_{\vert v_* \vert \leq 2N} \vert k_w (v,v_*, I,I_*) h(s,x-(t-s)v, v_*,I_*) \vert  dv_* dI_* ds\\ 
			&\leq Ce^{\nu_0 \tilde{t}} e^{-\frac{\nu_0}{2} t} \Vert h_0 \Vert_{L^\infty}  \\
			&\quad + e^{\nu_0 \tilde{t}}\int_0^t ds e^{-\frac{\nu_0}{2} (t-s)}\int_0^s ds'  e^{-\frac{\nu_0}{2} (s-s')} \iint_{ \vert v_* \vert \leq 2N,I_* \leq 2N} dv_* dI_*  \\
			&\quad \quad \times \iint_{\R^3 \times \R_+} dv_{**}dI_{**}\vert k_w (v,v_*, I,I_*) k_w (v_{*},v_{**},I_*,I_{**})h(s',x-(t-s)v-(s-s')v_*, v_{**},I_{**}) \vert  \\
			&\quad + e^{\nu_0 \tilde{t}}\int_0^t ds e^{-\frac{\nu_0}{2} (t-s)}\int_0^s ds'  e^{-\frac{\nu_0}{2} (s-s')} \iint_{\vert v_* \vert \leq 2N,I_* \leq 2N} dv_* dI_* \vert k_w (v,v_*, I,I_*) w \Gamma_+ (f,f)(s') \vert \\ 
			&: = I_1 + I_2 + I_3. 
		\end{split}
	\end{align}
	To estimate $I_2$, we similarly divide into several cases concerning $v_{**}$ and $I_{**}$. The following cases 
	\begin{align*}
		\{ \vert v_{**} \vert \geq 3N\} \quad \textrm{and} \quad \{ \vert v_{**} \vert \leq 3N, I_{**} \geq 3N\}
	\end{align*}
	can be treated in the same way as Case 2 and Case 3. Hence, we omit the above cases and focus on main case $\{ \vert v_{**} \vert \leq 3N, I_{**} \leq 3N \}$. Using Lemma \ref{K esti}, H\"{o}lder's inequality and the change of variables $y = x- (t-s)v - (s-s') v_*$, we obtain 
	\begin{align}\label{XI2}
		\begin{split}
			&I_2 \chi_{\{\vert v_{**} \vert \leq 3N, I_{**}\leq 3N\}} \\
			&\leq e^{\nu_0 \tilde{t}}\int_0^t ds e^{-\frac{\nu_0}{2}(t-s)}\left[\int_0^{s-\lambda_1} + \int_{s-\lambda_1}^s \right]ds' e^{-\frac{\nu_0}{2}(s-s')}\\
			&\quad \times \left(\int_{ \vert v_* \vert \leq 2N,\vert v_{**} \vert \leq 3N, I_* \leq 2N, I_{**} \leq 3N } \vert k_w (v,v_*, I,I_*) k_w (v_{*},v_{**},I_*,I_{**}) \vert ^2 dv_* dv_{**} dI_* dI_{**}\right)^{1/2}\\ 
			& \quad \times \left ( \int_{ \vert v_* \vert \leq 2N,\vert v_{**} \vert \leq 3N, I_* \leq 2N, I_{**} \leq 3N} \vert h(s' ,x-(t-s)v-(s-s')v_{*}, v_{**},I_{**})\vert ^2 dv_* dv_{**} dI_* dI_{**} \right)^{1/2}\\
			& \leq C_N e^{\nu_0 \tilde{t}} \int_0^t e^{-\frac{\nu_0}{2}(t-s)} \int_0^{s-\lambda_1} e^{-\frac{\nu_0}{2}(s-s')} \\
			&\quad \quad  \times \left ( \int_{y \in \T^3, \vert v_{**} \vert \leq 3N,I_{**} \leq 3N} \left(\frac{1+N(s-s')^3}{(s-s')^3}\right)\vert h(s' ,y, v_{**},I_{**})\vert ^2 dy dv_{**}dI_{**} \right)^{1/2} ds' ds\\
			&\quad + C \lambda_1 e^{\nu_0 \tilde{t}} \sup_{0\leq s \leq t}\Vert h(s) \Vert_{L^{\infty}} \\
			&\leq C_{N,\lambda_1} e^{\nu_0 \tilde{t}} \int_0^t \int_0^{s-\lambda_1} e^{-\frac{\nu_0}{2}(t-s')} \left( \int_{y \in \T^3, \vert v_{**} \vert \leq 3N,I_{**} \leq 3N} \vert f(s', y, v_{**},I_{**})\vert^2 dydv_{**}dI_{**}\right)^{1/2} ds'ds
			\\
			&\quad + C \lambda_1 e^{\nu_0 \tilde{t}} \sup_{0\leq s \leq t}\Vert h(s) \Vert_{L^{\infty}}. 
		\end{split}
	\end{align} 
	Using Lemma \ref{relative entropy lemma}, we compute the integral part of the last line in \eqref{XI2}
	\begin{align} \label{L2 entropy}
		\begin{split}
			&\int_{y \in \T^3, I_{**} \leq 3N, \vert v_{**}\vert \leq 3N} \vert f(s', y, v_{**},I_{**})\vert^2 dydI_{**}dv_{**} \\
			&=\left(\int_{y \in \T^3, \lambda_2 \leq I_{**} \leq 3N, \vert v_{**}\vert \leq 3N}+ \int_{y \in \T^3,  I_{**} \leq \lambda_2, \vert v_{**}\vert \leq 3N}\right) \vert f(s', y, v_{**},I_{**})\vert^2 dydI_{**}dv_{**} \\
			& \leq C\lambda_2 \sup_{0\leq s \leq t} \Vert h(s) \Vert _{L^\infty}^2 \\
			&\quad +\int_{y \in \T^3, \lambda_2 \leq I_{**} \leq 3N, \vert v_{**}\vert \leq 3N} \vert f(s', y, v_{**},I_{**})\vert^2 \left(\chi_{\{ \vert \sqrt{M_{**}}f_{**} \vert \leq M_{**}\}} + \chi_{\{\vert \sqrt{M_{**}}f_{**} \vert >M_{**}\}}\right)dydI_{**}dv_{**}\\
			&\leq C\lambda_2 \sup_{0\leq s \leq t} \Vert h(s) \Vert _{L^\infty}^2 + \mathcal{E}(F_0) \\
			&\quad+ \sup_{0\leq s \leq t} \Vert h(s) \Vert_{L^\infty} \int_{y \in \T^3, \lambda_2 \leq I_{**} \leq 3N, \vert v_{**}\vert \leq 3N} \frac{\sqrt{M_{**}}}{\sqrt{M_{**}}}\vert f(s', y, v_{**},I_{**})\vert \chi_{\{\vert \sqrt{M_{**}}f_{**} \vert > M_{**}\}}dydI_{**}dv_{**}\\ 
			& \leq C\lambda_2  \sup_{0\leq s \leq t} \Vert h(s) \Vert _{L^\infty}^2 + \mathcal{E}(F_0) + C_{N,\lambda_2}  \sup_{0\leq s \leq t} \Vert h(s) \Vert_{L^\infty}  \mathcal{E}(F_0).
		\end{split}
	\end{align}
	Thus, we can further estimate $I_2$ over $\{ \vert v_{**} \vert \leq 3N, I_{**} \leq 3N \}$: 
	\begin{align} \label{K I2}
		\begin{split}
			&I_2 \chi_{\{\vert v_{**} \vert \leq 3N, I_{**}\leq 3N\}} \\
			&\leq e^{\nu_0 \tilde{t}} \left[\left(C_{N,\lambda_1} \sqrt{\lambda_2} + C \lambda_1\right)  \sup_{0\leq s \leq t} \Vert h(s) \Vert_{L^{\infty}} +\sqrt{\mathcal{E}(F_0)}\left(C_{N,\lambda_1}+ C_{N,\lambda_1,\lambda_2} \sup_{0\leq s \leq t} \Vert h(s) \Vert_{L^{\infty}}^{1/2} \right)\right].
		\end{split}
	\end{align}
	Next, we use Lemma \ref{GE} to estimate $I_3$ in \eqref{kw case}, and then 
	\begin{align*}
		I_3 &\leq C e^{\nu_0 \tilde{t}}\int_0^t ds e^{-\frac{\nu_0}{2} (t-s)}\int_0^s ds'  e^{-\frac{\nu_0}{2} (s-s')} \iint_{\vert v_* \vert \leq 2N, I_* \leq 2N} dv_* dI_* \vert k_w (v,v_*, I,I_*) \vert \\ 
		& \quad \times \Vert wf(s') \Vert_{L^{\infty}} \left ( \int_{\R^3 \times \R_+} dv_{**} dI_{**} (1+ \vert v_{**} \vert + \sqrt{I_{**}})^{-2\beta+8}\vert h(s', x-(t-s)v-(s-s')v_{*}, v_{**},I_{**}) \vert ^2 \right)^{1/2}.
	\end{align*}
	Similar to the $K_w$ case, we focus on the main case $\{ \vert v_{**} \vert \leq 3N, I_{**} \leq 3N\}$. Using Lemma \ref{GE}, H\"{o}lder's inequality, the change of variables $y=x-(t-s)v - (s-s')v_*$, and \eqref{L2 entropy}, we get
	\begin{align} \label{K I3}
		\begin{split}
			&I_3 \chi_{\{\vert v_{**}\vert\leq 3N, I_{**}\leq 3N\}}\\
			&\leq C e^{\nu_0 \tilde{t}}\int_0^t ds e^{-\frac{\nu_0}{2} (t-s)}\left[\int_0^{s-\lambda_1} +\int_{s-\lambda_1}^s \right] ds'  e^{-\frac{\nu_0}{2} (s-s')} \iint_{\vert v_* \vert \leq 2N,I_* \leq 2N} dv_* dI_* \vert k_w (v,v_*, I,I_*) \vert \\ 
			& \quad \times \Vert h(s') \Vert_{L^{\infty}} \left ( \int_{\vert v_{**} \vert \leq 3N, I_{**}\leq 3N} dv_{**} dI_{**} (1+ \vert v_{**} \vert + \sqrt{I_{**}})^{-2\beta+8}\vert h(s', x-(t-s)v-(s-s')v_{*}, v_{**},I_{**}) \vert ^2 \right)^\frac{1}{2}\\
			& \leq  C_{N}e^{\nu_0 \tilde{t}} \sup_{0\leq s \leq t} \Vert h(s) \Vert_{L^{\infty}} \int_0^t ds e^{-\frac{\nu_0}{2}(t-s)}  \int_0^{s-\lambda_1} ds' e^{-\frac{\nu_0}{2}(s-s')} \\
			&\quad \quad \times \left( \int_{y \in \T^3, I_{**} \leq 3N, \vert v_{**}\vert \leq 3N} \left(\frac{1+N(s-s')^3}{(s-s')^3}\right)\vert f(s', y, v_{**},I_{**})\vert^2 dydI_{**}dv_{**}\right)^{1/2} \\
			&\quad + C \lambda_1 e^{\nu_0 \tilde{t}} \sup_{0\leq s \leq t}\Vert h(s) \Vert_{L^{\infty}}^2 \\ 
			&\leq e^{\nu_0 \tilde{t}} \left[ \left(C_{N,\lambda_1}\sqrt{\lambda_2} +C \lambda_1 \right)\sup_{0\leq s \leq t} \Vert h(s) \Vert_{L^{\infty}} ^2 + \sqrt{\mathcal{E}(F_0)} \left(C_{N,\lambda_1}\sup_{0\leq s \leq t} \Vert h(s) \Vert_{L^{\infty}} +C_{N,\lambda_1,\lambda_2} \sup_{0\leq s \leq t} \Vert h(s) \Vert_{L^{\infty}}^{3/2}\right)\right].
		\end{split}
	\end{align}
	Combining \eqref{K case1}, \eqref{K case2}, \eqref{K case3}, \eqref{kw case}, \eqref{K I2}, and \eqref{K I3}, we have
	\begin{align} \label{K final esti}
		\begin{split}
			&\int_0^t e^{\frac{\nu_0}{2}\tilde{t}} e^{-\frac{\nu_0}{2}(t-s)} \vert K_w h (s) \vert  ds\\ &\leq Ce^{\nu_0 \tilde{t}} e^{-\frac{\nu_0}{2} t} \Vert h_0 \Vert_{L^\infty} \\ 
			&\quad + e^{\nu_0\tilde{t}} \left(\sup_{0\leq s \leq t} \Vert h(s) \Vert_{L^{\infty}} + \sup_{0\leq s \leq t} \Vert h(s) \Vert_{L^{\infty}}^2 \right) \left[ \frac{C}{N^{1/8}} +\left(C_{N,\lambda_1}\sqrt{\lambda_2} + C\lambda_1 \right) \right]\\
			& \quad +e^{\nu_0 \tilde{t}} \sqrt{\mathcal{E}(F_0)}\left[ \left\{C_{N,\lambda_1}\left(1+\sup_{0\leq s \leq t} \Vert h(s) \Vert_{L^{\infty}}\right) +C_{N,\lambda_1,\lambda_2} \left(\sup_{0\leq s \leq t} \Vert h(s) \Vert_{L^{\infty}}^{1/2}+\sup_{0\leq s \leq t} \Vert h(s) \Vert_{L^{\infty}}^{3/2}\right)\right\}\right].
		\end{split}
	\end{align}
	From now on, we consider the part $w\Gamma_+(f,f)$. We apply Lemma \ref{GE} to obtain the estimate 
	\begin{align}\label{GG}
		\begin{split}
			&\int_{0}^{t}e^{\frac{\nu_0}{2}\tilde{t}}e^{-\frac{\nu_0}{2}(t-s)}|w(v,I)\Gamma_+(f,f)(s,x-(t-s)v,v,I)|ds \cr
			&\leq \int_{0}^{t}e^{\nu_0\tilde{t}}e^{-\frac{\nu_0}{2}(t-s)}\frac{C\|h(s)\|_{L^\infty}}{1+|v|+I^{1/4}}\left(\int_{\R^3 \times \R_{+}}\left(1+|\eta|+\sqrt{L}\right)^{-2\beta+8}|h(s,x-(t-s)v,\eta,L)|^2d\eta dL \right)^\frac{1}{2}ds.
		\end{split}
	\end{align}
	To further estimate \eqref{GG}, we separate into two parts as follows:
	\begin{align*}
		(1)\; &|\eta|\geq N \quad \text{or} \quad L\geq N, \cr
		(2)\; &|\eta|\leq N, L\leq N.
	\end{align*}
	For case $(1)$, by direct computation, we have
	\begin{align}\label{GGE1}
		\begin{split}
			&\int_{0}^{t}e^{\nu_0\tilde{t}}e^{-\frac{\nu_0}{2}(t-s)}|w(v,I)\Gamma_+(f,f)(s,x-(t-s)v,v,I)|ds\cr
			&\leq C\sup_{0\leq s \leq t}\|h(s)\|^2_{L^\infty}\int_{0}^{t}e^{\nu_0\tilde{t}}e^{-\frac{\nu_0}{2}(t-s)}\left(\int_{\R^3 \times \R_+}\frac{1}{\left(1+|\eta|+\sqrt{L}\right)^{2\beta-8}} \left(\chi_{\{|\eta| \geq N\}}+\chi_{\{L \geq N\}}\right)d\eta dL\right)^\frac{1}{2}ds \cr\\
			&\leq Ce^{\nu_0\tilde{t}}\sup_{0\leq s \leq t}\|h(s)\|^2_{L^\infty}\int_{0}^{t}dse^{-\frac{\nu_0}{2}(t-s)}\\
			&\quad \times \left(\int_{\R^3 \times \R_+}\frac{1}{\left(1+|\eta|+\sqrt{L}\right)^{2\beta-14}\Big(1+|\eta|\Big)^{7/2}\left(1+\sqrt{L}\right)^{5/2}} \left(\chi_{\{|\eta| \geq N\}}+\chi_{\{L \geq N\}}\right)d\eta dL\right)^\frac{1}{2} \\
			&\leq \frac{e^{\nu_0\tilde{t}}C_{\beta,\nu_0}}{N^{1/4}}\sup_{0\leq s \leq t}\|h(s)\|^2_{L^\infty},
		\end{split}
	\end{align}
	where the last inequality comes from $\beta > 7$. 
	To estimate case $(2)$, we use Duhamel's principle to get

	\begin{align}\label{GDE}
		\begin{split}
			\left(\int_{|\eta|\leq N, L\leq N}\left(1+|\eta|+\sqrt{L}\right)^{-2\beta+8}|h(s,x-(t-s)v,\eta,L)|^2d\eta dL\right)^{1/2}\leq II_1+II_2+II_3.
		\end{split}
	\end{align}
	Here, $II_1$, $II_2$, and $II_3$ are defined by
	{\footnotesize
		\begin{align*}
			II_1&:=e^{\nu_0\tilde{t}}\left(\int_{|\eta|\leq N,L\leq N}\left(1+|\eta|+\sqrt{L}\right)^{-2\beta+8}\left|e^{-\frac{\nu_0}{2}s} \| h_0 \|_{L^\infty} \right|^2d\eta dL\right)^\frac{1}{2},\cr
			II_2&:=e^{\nu_0\tilde{t}}\left(\int_{|\eta|\leq N, L\leq N}\left(1+|\eta|+\sqrt{L}\right)^{-2\beta+8}\left|\int_{s-\lambda_{1}}^{s}e^{-\frac{\nu_0}{2}(s-\tau)}\left|[K_wh+w\Gamma_+(f,f)](\tau,\hat{X}(\tau),\eta,L)\right|d\tau\right|^2d\eta dL\right)^\frac{1}{2},\cr
			II_3&:=e^{\nu_0\tilde{t}}\left(\int_{|\eta|\leq N, L\leq N}\left(1+|\eta|+\sqrt{L}\right)^{-2\beta+8}\left|\int_{0}^{s-\lambda_{1}}e^{-\frac{\nu_0}{2}(s-\tau)}\left|[K_wh+w\Gamma_+(f,f)](\tau,\hat{X}(\tau),\eta,L)\right|d\tau\right|^2d\eta dL\right)^\frac{1}{2},
		\end{align*}
	}
	where the $\hat{X}(\tau):=x-(t-s)v-(s-\tau)\eta$. For $II_1$ and $II_2$, it follows from Lemma \ref{K esti} and Lemma \ref{nonlinear esti} that 
	\begin{align}\label{GDEI1}
		\begin{split}
			II_1&\leq e^{\nu_0\tilde{t}}e^{-\frac{\nu_0}{2}s}\Vert h_0 \Vert_{L^\infty} \left(\int_{|\eta|\leq N, L\leq N}\left(1+|\eta|+\sqrt{L}\right)^{-2\beta+8}d\eta dL\right)^\frac{1}{2}\cr
			&\leq C_{\beta,\nu_0}e^{\nu_0\tilde{t}}e^{-\frac{\nu_0}{2}s}\Vert h_0 \Vert_{L^\infty},
		\end{split}
	\end{align}
	and
	\begin{align}\label{GDEI2}
		\begin{split}
			II_2\leq C_{\beta,\nu_0}e^{\nu_0\tilde{t}}\lambda_1\left[\sup_{0\leq \tau \leq s}\|h(\tau)\|_{L^\infty}+\sup_{0\leq \tau \leq s}\|h(\tau)\|^2_{L^\infty}\right].
		\end{split}
	\end{align}
	For $II_3$, we consider the following two separate parts:
	\begin{align}\label{II31}
		II_{31}:=e^{\nu_0\tilde{t}}\left(\int_{|\eta|\leq N, L\leq N}\left(1+|\eta|+\sqrt{L}\right)^{-2\beta+8}\left|\int_{0}^{s-\lambda_{1}}e^{-\frac{\nu_0}{2}(s-\tau)}\left|K_wh(\tau,\hat{X}(\tau),\eta,L)\right|d\tau\right|^2d\eta dL\right)^\frac{1}{2},
	\end{align}
	and
	{\small
		\begin{align}\label{II32}
			II_{32}:=e^{\nu_0\tilde{t}}\left(\int_{|\eta|\leq N, L\leq N}\left(1+|\eta|+\sqrt{L}\right)^{-2\beta+8}\left|\int_{0}^{s-\lambda_{1}}e^{-\frac{\nu_0}{2}(s-\tau)}\left|w\Gamma_+(f,f)(\tau,\hat{X}(\tau),\eta,L)\right|d\tau\right|^2d\eta dL\right)^\frac{1}{2}.
		\end{align}
	}
	For $II_{31}$, we separate into two parts in the following way:
	\begin{align}\label{GDEkwS}
		\begin{split}
			&\int_{0}^{s-\lambda_{1}}e^{-\frac{\nu_0}{2}(s-\tau)}\left|K_wh(\tau,\hat{X}(\tau),\eta,L)\right|d\tau\cr
			&\leq\int_{0}^{s-\lambda_{1}}e^{-\frac{\nu_0}{2}(s-\tau)}\int_{\R^3\times\R_+}\left|k_w(\eta,\eta_*,L,L_*)h(\tau,\hat{X}(\tau),\eta_*,L_*)\right|d\eta_*dL_*d\tau\cr
			&\leq\int_{0}^{s-\lambda_{1}}e^{-\frac{\nu_0}{2}(s-\tau)}\int_{\R^3 \times \R_+}\left |k_w(\eta,\eta_*,L,L_*)h(\tau,\hat{X}(\tau),\eta_*,L_*)\right|\left(\chi_{\{| \eta_* | \geq 2N\}}+\chi_{\{L_*  \geq 2N\}}\right)d\eta_*dL_*d\tau\cr\\
			&\quad + \int_{0}^{s-\lambda_{1}}e^{-\frac{\nu_0}{2}(s-\tau)}\int_{|\eta_*|\leq 2N,L_*\leq 2N} \left| k_w(\eta,\eta_*,L,L_*)h(\tau,\hat{X}(\tau),\eta_*,L_*)\right| d\eta_* dL_* d\tau \\ 
			&\leq \frac{C_{\nu_0}}{N^{1/8}}\sup_{0\leq \tau \leq s}\|h(\tau)\|_{L^\infty}\\
			&\quad+\int_{0}^{s-\lambda_{1}}e^{-\frac{\nu_0}{2}(s-\tau)}\int_{|\eta_*|\leq 2N,L_*\leq 2N}\left|k_w(\eta,\eta_*,L,L_*)h(\tau,\hat{X}(\tau),\eta_*,L_*)\right|d\eta_*dL_*d\tau,
		\end{split}
	\end{align}
	where the last inequality could be obtained by Lemma \ref{K esti} and $|\eta- \eta_*| \geq N$ under $| \eta | \leq N$. 
	Inserting \eqref{GDEkwS} into \eqref{II31} gives the following estimate for $II_{31}$: 
	{\small
		\begin{align}\label{II31-A}
			\begin{split}
				II_{31}&\leq  \frac{C_{\beta,\nu_0}e^{\nu_0\tilde{t}}}{N^{1/8}}\sup_{0\leq \tau \leq s}\|h(\tau)\|_{L^\infty}\cr
				&\quad+e^{\nu_0\tilde{t}}\left(\int_{|\eta|\leq N,L\leq N}\left|\int_{0}^{s-\lambda_{1}}e^{-\frac{\nu_0}{2}(s-\tau)}\int_{|\eta_*|\leq 2N,L_*\leq 2N}k_w(\eta,\eta_*,L,L_*)h(\tau,\hat{X}(\tau),\eta_*,L_*)d\eta_*dL_*d\tau\right|^2d\eta dL\right)^\frac{1}{2}.
			\end{split}
		\end{align}
	}
	The second term in the right-hand side of \eqref{II31-A} can be estimated by H\"{o}lder's  inequality and Lemma \ref{K esti} as follows:
	\begin{align}\label{KL2}
		\begin{split}
			&e^{\nu_0\tilde{t}}\left(\int_{|\eta|\leq N,L\leq N}\left|\int_{0}^{s-\lambda_{1}}e^{-\frac{\nu_0}{2}(s-\tau)}\int_{|\eta_*|\leq 2N,L_*\leq 2N}k_w(\eta,\eta_*,L,L_*)h(\tau,\hat{X}(\tau),\eta_*,L_*)d\eta_*dL_*d\tau\right|^2d\eta dL\right)^\frac{1}{2}\cr
			&\leq C_{\beta,\nu_0}e^{\nu_0\tilde{t}}\left(\int_{|\eta|\leq N,L\leq N}\int_{0}^{s-\lambda_{1}}e^{-\frac{\nu_0}{2}(s-\tau)}\int_{|\eta_*|\leq 2N,L_*\leq 2N}\left|k_w(\eta,\eta_*,L,L_*)\right|^2d\eta_*dL_*d\eta dLd\tau\right)^\frac{1}{2}\cr
			&\quad \times
			\left(\int_{|\eta|\leq N,L\leq N}\int_{0}^{s-\lambda_{1}}e^{-\frac{\nu_0}{2}(s-\tau)}\int_{|\eta_*|\leq 2N,L_*\leq 2N}|f(\tau,\hat{X}(\tau),\eta_*,L_*)|^2d\eta_*dL_*d\eta dLd\tau\right)^\frac{1}{2}\\
			&\leq C_{\beta,\nu_0,N}e^{\nu_0\tilde{t}}
			\left(\int_{0}^{s-\lambda_{1}}e^{-\frac{\nu_0}{2}(s-\tau)}\int_{|\eta|\leq N,L\leq N}\int_{|\eta_*|\leq 2N,L_*\leq 2N}|f(\tau,\hat{X}(\tau),\eta_*,L_*)|^2d\eta_*dL_*d\eta dLd\tau\right)^\frac{1}{2}.
		\end{split}
	\end{align}
	To complete the estimate, we use the change of variables $y:=\hat{X}(\tau)=x-(t-s)v-(s-\tau)\eta$ and the same method with \eqref{F2}, \eqref{KL2}. Thus, we have 
	\begin{align}\label{KL23}
		\begin{split}
			&C_{\beta,\nu_0,N}e^{\nu_0\tilde{t}}\left(\int_{0}^{s-\lambda_{1}}e^{-\frac{\nu_0}{2}(s-\tau)}\int_{|\eta|\leq N,L\leq N}\int_{|\eta_*|\leq 2N,L_*\leq 2N}|f(\tau,\hat{X}(\tau),\eta,L)|^2d\eta_*dL_*d\eta dLd\tau\right)^\frac{1}{2}\cr
			&\leq C_{\beta,\nu_0,N}e^{\nu_0\tilde{t}}\left[\left(\frac{1}{\lambda^3_{1}}+N\right)\left\{\lambda_{2}\sup_{0\leq \tau\leq s}\|h(\tau)\|^2_{L^\infty}+\mathcal{E}(F_0)+C\frac{e^{2N}}{\lambda_{2}^{(\delta-2)/4}}  \sup_{0\leq \tau \leq s} \Vert h(\tau) 	\Vert_{L^\infty}  \mathcal{E}(F_0)\right\}\right]^{\frac{1}{2}}.
		\end{split}
	\end{align}
	Hence, $II_{31}$ can be further controlled by 
	{\small
		\begin{align}\label{I3}
			\begin{split}
				II_{31}&\leq\frac{ C_{\beta,\nu_0}e^{\nu_0\tilde{t}}}{N^{1/8}}\sup_{0\leq \tau\leq s}\|h(\tau)\|_{L^\infty}\cr&\quad+C_{\beta,\nu_0,N}e^{\nu_0\tilde{t}}\left[\left(\frac{1}{\lambda^3_{1}}+N\right)\left\{\lambda_{2}\sup_{0\leq \tau\leq s}\|h(\tau)\|^2_{L^\infty}+\mathcal{E}(F_0)+C\frac{e^{2N}}{\lambda_{2}^{(\delta-2)/4}}  \sup_{0\leq \tau \leq s} \Vert h(\tau) 	\Vert_{L^\infty}  \mathcal{E}(F_0)\right\}\right]^{\frac{1}{2}}.
			\end{split}
		\end{align}
	}
	Similarly, for $II_{32}$ in \eqref{II32}, we can deal with the following terms by using Lemma \ref{GE} 
	\begin{align*}
		&\left|\int_{0}^{s-\lambda_{1}}e^{-\frac{\nu_0}{2}(s-\tau)}|w\Gamma_+(f,f)(\tau,\hat{X}(\tau),\eta,L)|d\tau\right|^2\cr
		&\leq		C\left |\int_{0}^{s-\lambda_{1}}e^{-\frac{\nu_0}{2}(s-\tau)}\|h(\tau)\|_{L^\infty}\left(\int_{\R^3 \times \R_+}\left(1+|\eta_*|+\sqrt{L_*}\right)^{-2\beta+8}|h(\tau,\hat{X}(\tau),\eta_*,L_*)|^2d\eta_*dL_*\right)^{1/2}d\tau\right|^2.
	\end{align*}
	Additionally, this term is divided into two parts similar to \eqref{GDEkwS}. Applying the method \eqref{GGE1} for $\{|\eta_*|\geq 2N \} \cup\{L_*\geq 2N \}$ parts yields that 
	{\small
		\begin{align}\label{I3G}
			\begin{split}
				II_{32}&\leq \frac{C_{\beta,\nu_0}e^{\nu_0\tilde{t}}}{N^{1/4}}\sup_{0\leq \tau \leq s}\|h(\tau)\|^2_{L^\infty}\cr
				&\quad+Ce^{\nu_0\tilde{t}}\sup_{0 \leq \tau \leq s }\|h(\tau)\|_{L^\infty}\left(\int_{0}^{s-\lambda_{1}}e^{-\frac{\nu_0}{2}(s-\tau)}\int_{|\eta|\leq N, L\leq N}\int_{|\eta_*|\leq 2N,L_*\leq 2N}|h(\tau,\hat{X}(\tau),\eta_*,L_*)|^2 d\eta_*dL_*d\eta dL d\tau\right)^\frac{1}{2}.
			\end{split}
		\end{align}}
	Using the change of variables $y:=\hat{X}(\tau)=x-(t-s)v-(s-\tau)\eta$ and the same method as \eqref{F2}, \eqref{I3G}, $II_{32}$ can be estimated by
	\begin{align}\label{I3GR}
		\begin{split}
			II_{32}&\leq \frac{C_{\beta,\nu_0}e^{\nu_0\tilde{t}}}{N^{1/4}}\sup_{0\leq \tau \leq s}\|h(\tau)\|^2_{L^\infty}+C_{\nu_0}e^{\nu_0\tilde{t}}\sup_{0\leq \tau \leq s}\|h(\tau)\|_{L^\infty}\\
			&\quad \times \left[\left(\frac{1}{\lambda^3_{1}}+N\right)\left\{\lambda_{2}\sup_{0\leq \tau\leq s}\|h(\tau)\|^2_{L^\infty}+\mathcal{E}(F_0)+C\frac{e^{2N}}{\lambda_{2}^{(\delta-2)/4}}  \sup_{0\leq \tau \leq s} \Vert h(\tau) 	\Vert_{L^\infty}  \mathcal{E}(F_0)\right\}\right]^{\frac{1}{2}}.
		\end{split}
	\end{align}
	Combining \eqref{I3} and \eqref{I3GR}, we can estimate $II_3$ as shown below:
	{\small
		\begin{align}\label{GDEI3}
			\begin{split}
				II_3&\leq \frac{C_{\beta,\nu_0}e^{\nu_0\tilde{t}}}{N^{1/8}}\left(\sup_{0\leq \tau \leq s}\|h(\tau)\|_{L^\infty}+\sup_{0\leq \tau \leq s}\|h(\tau)\|^2_{L^\infty}\right)\cr
				&\quad+ C_{\beta,\nu_0}e^{\nu_0\tilde{t}}\left\{\left(\frac{1}{\lambda^3_{1}}+N\right)\left(\lambda_{2}\sup_{0\leq s\leq t}\|h(s)\|^2_{L^\infty}+\mathcal{E}(F_0)+C\frac{e^{2N}}{\lambda_{2}^{(\delta-2)/4}}  \sup_{0\leq s \leq t} \Vert h(s) 	\Vert_{L^\infty}  \mathcal{E}(F_0)\right)\right\}^{\frac{1}{2}}\cr
				&\quad+ C_{\nu_0}e^{\nu_0\tilde{t}}\sup_{0\leq \tau \leq s}\|h(\tau)\|_{L^\infty}\\
				&\quad \quad \times \left\{\left(\frac{1}{\lambda^3_{1}}+N\right)\left(\lambda_{2}\sup_{0\leq s\leq t}\|h(s)\|^2_{L^\infty}+\mathcal{E}(F_0)+C\frac{e^{2N}}{\lambda_{2}^{(\delta-2)/4}}  \sup_{0\leq s \leq t} \Vert h(s) 	\Vert_{L^\infty}  \mathcal{E}(F_0)\right)\right\}^{\frac{1}{2}}.
			\end{split}
		\end{align}
	}
	By \eqref{GDEI1}, \eqref{GDEI2}, and \eqref{GDEI3}, we get the estimate for \eqref{GDE}. If we insert the estimate for \eqref{GGE1} and \eqref{GDE} into \eqref{GG}, then we control the part $w\Gamma_+(f,f)$ in this way:
	\begin{align} \label{GR}
		\begin{split}
			&\int_{0}^{t}e^{\nu_0\tilde{t}}e^{-\frac{\nu_0}{2}(t-s)}|w(v,I)\Gamma_+(f,f)(s,x-(t-s)v,v,I)|ds\cr
			&\leq C_{\beta,\nu_0}e^{\nu_0\tilde{t}}e^{-\frac{\nu_0}{2}t}\Vert h_0 \Vert_{L^\infty}\int_{0}^{t}\|h(s)\|_{L^\infty}ds\cr
			&\quad+C_{\beta,\nu_0}e^{\frac{5\nu_0}{4}\tilde{t}} \lambda_1 \left(\sup_{0\leq s \leq t}\|h(s)\|^2_{L^\infty}+\sup_{0\leq s\leq t}\|h(s)\|^3_{L^\infty}\right)+\frac{C_{\beta,\nu_0}e^{\nu_0\tilde{t}}}{N^{1/8}}\left(\sup_{0\leq s \leq t}\|h(s)\|^2_{L^\infty}+\sup_{0\leq s \leq t}\|h(s)\|^3_{L^\infty}\right)\cr
			&\quad+C_{N,\nu_0}e^{\frac{\nu_0}{4}\tilde{t}}\left(\sup_{0\leq s \leq t}\|h(s)\|_{L^\infty}+\sup_{0\leq s \leq t}\|h(s)\|^2_{L^\infty}\right)\\
			&\quad \quad \times \left\{\left(\frac{1}{\lambda^3_{1}}+N\right)\left(\lambda_{2}\sup_{0\leq s\leq t}\|h(s)\|^2_{L^\infty}+\mathcal{E}(F_0)+C\frac{e^{2N}}{\lambda_{2}^{(\delta-2)/4}}  \sup_{0\leq s \leq t} \Vert h(s) 	\Vert_{L^\infty}  \mathcal{E}(F_0)\right)\right\}^{\frac{1}{2}} \\ 
			&\leq C_{\beta,\nu_0}e^{\nu_0\tilde{t}}e^{-\frac{\nu_0}{2}t}\Vert h_0 \Vert_{L^\infty}\int_{0}^{t}\|h(s)\|_{L^\infty}ds\cr
			&\quad+C_{\beta,\nu_0}e^{\frac{5\nu_0}{4}\tilde{t}}\left(\sup_{0\leq s \leq t}\|h(s)\|^2_{L^\infty}+\sup_{0\leq s \leq t}\|h(s)\|^3_{L^\infty}\right)\left[\lambda_{1}+\frac{1}{N}+C_{N,\lambda_{1}}\sqrt{\lambda_{2}}\right]\cr
			&\quad+C_{N,\nu_0}e^{\nu_0\tilde{t}}\sqrt{\mathcal{E}(F_0)}\left(\sup_{0\leq s \leq t}\|h(s)\|_{L^\infty}+\sup_{0\leq s \leq t}\|h(s)\|^2_{L^\infty}\right)\left\{C_{N,\lambda_{1}}\left(1+C_{N,\lambda_{2}} \sup_{0\leq s \leq t} \Vert h(s) 	\Vert_{L^\infty}  \right)\right\}^{\frac{1}{2}}.
		\end{split}
	\end{align}
	In conclusion, from \eqref{duhamel},\eqref{K final esti}, and \eqref{GR}, it holds that:
	\begin{align*}
		|h(t,x,v,I)|&\leq \bar{C}e^{-\frac{\nu_0}{2}(t-\tilde{t})}\Vert h_0 \Vert_{L^\infty}\left(1+\int_{0}^{t}\|h(s)\|_{L^\infty}ds\right)\cr
		&\quad+\bar{C}e^{\frac{5\nu_0}{4}\tilde{t}}\left(\sup_{0\leq s \leq t}\|h(s)\|_{L^\infty}+\sup_{0\leq s \leq t}\|h(s)\|^2_{L^\infty}+\sup_{0\leq s \leq t}\|h(s)\|^3_{L^\infty}\right)\left[\lambda_{1}+\frac{1}{N}+C_{N,\lambda_{1}}\sqrt{\lambda_{2}}\right]\cr
		&\quad+\bar{C}e^{\nu_0 \tilde{t}} \sqrt{\mathcal{E}(F_0)}\\
		&\quad \quad \times \left[ \left\{C_{N,\lambda_1}\left(1+\sup_{0\leq s \leq t} \Vert h(s) \Vert_{L^{\infty}}\right) +C_{N,\lambda_1,\lambda_2} \left(\sup_{0\leq s \leq t} \Vert h(s) \Vert_{L^{\infty}}^{1/2}+\sup_{0\leq s \leq t} \Vert h(s) \Vert_{L^{\infty}}^{3/2}\right)\right\}\right]\cr
		&\quad+\bar{C}e^{\nu_0\tilde{t}}\sqrt{\mathcal{E}(F_0)}\\
		&\quad \quad \times \left(\sup_{0\leq s \leq t}\|h(s)\|_{L^\infty}+\sup_{0\leq s \leq t}\|h(s)\|^2_{L^\infty}\right)\left\{C_{N,\lambda_{1}}\left(1+C_{N,\lambda_{2}} \sup_{0\leq s \leq t} \Vert h(s) 	\Vert_{L^\infty}\right)\right\}^{\frac{1}{2}},
	\end{align*}
	where the constant $\bar{C}:=\max\{1,C_{\beta,\nu_0},C_{N,\nu_0}\}$. Therefore, the proof is completed.
\end{proof}

\section{Local-in-time existence in $L^{\infty}$}
\begin{lemma} \cite{DLpoly} \label{local existence} 
	Let $\beta>7$ in the velocity weight function $w(v,I)$. Suppose $F_0 = M(v,I) + \sqrt{M(v,I)} f_0 (x,v,I)$ $\geq 0$ and $\Vert wf_0 \Vert_{L^{\infty}} < \infty$. Then, there exists a positive time $\hat{t}_0$ such that the initial value problem \eqref{poly be} on the polyatomic Boltzmann equation has a unique solution $F(t,x,v,I) = M(v,I) + \sqrt{M(v,I)} f(t,x,v,I) \geq 0$ on $t \in [0,\hat{t}_0]$, satisfying 
	\begin{align*}
		\sup_{0\leq t \leq \hat{t}_0} \Vert wf(t) \Vert_{L^{\infty}} \leq 2 \Vert wf_0 \Vert_{L^{\infty}}.
	\end{align*} 
\end{lemma}
\section {Global-in-time existence in $L^{\infty}$}
\begin{theorem} \cite{DLpoly} \label{small global}
	Assume the same conditions in Lemma \ref{local existence}. Suppose the conservation laws \eqref{laws} with $(M_I,J_I,E_I)=(0,0,0)$. 
	Then, there are constants $\kappa >0$ and $\vartheta>0$ such that if $\| wf_0 \| \leq \kappa$, the polyatomic Boltzmann equation \eqref{poly be} has a unique global mild solution $F(t,x,v,I) = M(v,I)+ \sqrt{M(v,I)} f(t,x,v,I) \geq 0$ satisfying 
	\begin{align*}
		\| wf(t) \|_{L^{\infty}} \leq C e^{-\vartheta t} \| wf_0 \|_{L^{\infty}}, \quad \forall t \geq 0,
	\end{align*}
	where $\vartheta>0$ is a positive constant.  
\end{theorem}
\noindent \textbf{Proof of Theorem \ref{main thm}} 
For simplicity of notation, we restate \eqref{main estimate} in Lemma \ref{L infty est} as 
\begin{align} \label{main revised}
	\Vert h(t) \Vert _{L^{\infty}} \leq  \bar{C}e^{-\frac{\nu_0}{2}(t-\tilde{t})}\Vert h_0 \Vert_{L^\infty}\left(1+\int_{0}^{t}\|h(s)\|_{L^\infty}ds\right) +D,  
\end{align}
where 
\begin{align} \label{D}
	\begin{split}
		D&:= \bar{C}e^{\frac{5\nu_0}{4}\tilde{t}}\left(\sup_{0\leq s \leq t}\|h(s)\|_{L^\infty}+\sup_{0\leq s \leq t}\|h(s)\|^2_{L^\infty}+\sup_{0\leq s \leq t}\|h(s)\|^3_{L^\infty}\right)\left[\lambda_{1}+\frac{1}{N}+C_{N,\lambda_{1}}\sqrt{\lambda_{2}}\right]\cr
		&\quad +\bar{C}e^{\nu_0 \tilde{t}} \sqrt{\mathcal{E}(F_0)}\left[ \left\{C_{N,\lambda_1}\left(1+\sup_{0\leq s \leq t} \Vert h(s) \Vert_{L^{\infty}}\right) +C_{N,\lambda_1,\lambda_2} \left(\sup_{0\leq s \leq t} \Vert h(s) \Vert_{L^{\infty}}^{1/2}+\sup_{0\leq s \leq t} \Vert h(s) \Vert_{L^{\infty}}^{3/2}\right)\right\}\right]\cr
		&\quad +\bar{C}e^{\nu_0\tilde{t}}\sqrt{\mathcal{E}(F_0)}\left(\sup_{0\leq s \leq t}\|h(s)\|_{L^\infty}+\sup_{0\leq s \leq t}\|h(s)\|^2_{L^\infty}\right)\left\{C_{N,\lambda_{1}}\left(1+C_{N,\lambda_{2}} \sup_{0\leq s \leq t} \Vert h(s) 	\Vert_{L^\infty}\right)\right\}^{\frac{1}{2}}. 
	\end{split}
\end{align}
If we define 
\begin{align*}
	Z(t):= 1+ \int_0^t \Vert h(s) \Vert_{L^{\infty}} ds, 
\end{align*}
then it follows from \eqref{main revised} that 
\begin{align*}
	Z'(t) \leq \bar{C}e^{-\frac{\nu_0}{2}(t-\tilde{t})}M_0 Z(t) +D, 
\end{align*}
whenever $\Vert h_0 \Vert_{L^{\infty}} \leq M_0$. 
Multiplying each side of the equation above by $\exp\left \{ - \frac{2\bar{C}}{\nu_0} e^{\frac{\nu_0}{2} \tilde{t}} M_0 (1-e^{-\frac{\nu_0}{2}t})\right \}$, we obtain 
\begin{align*}
	\left(Z(t)\exp\left \{ - \frac{2\bar{C}}{\nu_0} e^{\frac{\nu_0}{2} \tilde{t}} M_0 (1-e^{-\frac{\nu_0}{2}t})\right \} \right)' \leq D, \quad \forall t \in [0,T_0].
\end{align*}
Integrating over the time interval $[0,t]$, we get 
\begin{align} \label{proof est 1}
	\begin{split}
		Z(t) &\leq (1+Dt) \exp\left \{  \frac{2\bar{C}}{\nu_0} e^{\frac{\nu_0}{2} \tilde{t}} M_0 (1-e^{-\frac{\nu_0}{2}t}) \right \} \\
		&\leq (1+Dt)\exp\left \{  \frac{2\bar{C}}{\nu_0} e^{\frac{\nu_0}{2} \tilde{t}} M_0  \right \}. 
	\end{split}
\end{align}
Next, substituting \eqref{proof est 1} into \eqref{main revised} yields that 
\begin{align} \label{proof est 2}
	\Vert h(t) \Vert_{L^{\infty}} \leq \bar{C}M_0\exp\left \{ \frac{\nu_0}{2} \tilde{t}+ \frac{2\bar{C}}{\nu_0} e^{\frac{\nu_0}{2} \tilde{t}}M_0\right\} (1+Dt) e^{-\frac{\nu_0}{2}t} + D, 
\end{align}
for all $0\leq t \leq T_0$. We define 
\begin{align} \label{Mbar}
	\bar{M} := 4\bar{C}M_0\exp\left \{ \frac{\nu_0}{2} \tilde{t}+ \frac{2\bar{C}}{\nu_0} e^{\frac{\nu_0}{2} \tilde{t}}M_0\right\},
\end{align}
and 
\begin{align} \label{T0}
	T_0 := \frac{4}{\nu_0} \ln \frac{\bar{M}}{\kappa},
\end{align}
for $0<\kappa <1$. Using \eqref{proof est 2} and the definition of $\bar{M}$ in \eqref{Mbar}, we have 
\begin{align*}
	\Vert h(t) \Vert_{L^\infty} \leq \frac{1}{4} \bar{M} (1+Dt) e^{-\frac{\nu_0}{2} t} +D  \leq \frac{1}{4} \bar{M} (1+2D) e^{-\frac{\nu_0}{4}t} +D, 
\end{align*}
where we have used $t e^{-\frac{\nu_0}{4} t} \leq 2$. In \eqref{D}, we first take $\lambda_1 =\lambda_1(\bar{M})>0$ small enough, then $N= N(\lambda_1, \bar{M})>0$ sufficiently large, and then $\lambda_2 = \lambda_2 (\lambda_1 ,N, \bar{M})>0$ sufficiently small. Finally, let $\mathcal{E}(F_0) \leq \varepsilon = \varepsilon(\lambda_1,\lambda_2,N, \bar{M})$ sufficently small, so that 
\begin{align*}
	D \leq \min\left \{\frac{\bar{M}}{8}, \frac{\kappa}{4}, \frac{1}{4} \right\}.
\end{align*} 
Therefore, 
\begin{align} \label{proof est 3}
	\Vert h(t) \Vert_{L^{\infty}} \leq \frac{3}{8} \bar{M}+ \frac{1}{8} \bar{M}= \frac{1}{2} \bar{M},
\end{align}
for all $0\leq t \leq T_0$, which implies that a priori assumption \eqref{HC} is closed.

Next, we will extend the solution to the Boltzmann equation to the time interval $[0,T_0]$ using \eqref{proof est 3} and Lemma \ref{local existence}. By Lemma \ref{local existence} and the definition \eqref{Mbar} of $\bar{M}$, the solution to the Boltzmann equation exists on the interval $[0,\hat{t}_0]$ and satisfies 
\begin{align*}
	\sup_{0\leq t \leq \hat{t}_0} \Vert h(t) \Vert_{L^{\infty}} \leq 2 \Vert h_0 \Vert_{L^{\infty}} \leq \frac{1}{2} \bar{M}. 
\end{align*}
Setting $\hat{t}_0$ as the initial time, Lemma \ref{local existence} provides the local existence time $\tilde{t} = \hat{t}_0 (\bar{M}/2)$ satisfying
\begin{align*}
	\sup_{\hat{t}_0 \leq t \leq \hat{t}_0 + \tilde{t}} \Vert h(t) \Vert_{L^{\infty}} \leq 2 \Vert h(\hat{t}_0)\Vert_{L^{\infty}} \leq \bar{M}. 
\end{align*}
Since a priori assumption \eqref{HC} holds for $t \in [0, \hat{t}_0 + \tilde{t}]$, the estimate \eqref{proof est 3} is also satisfied for $t \in [0, \hat{t}_0  +\tilde{t}]$. Hence, it follows from \eqref{proof est 3} that 
\begin{align*}
	\sup_{0\leq t \leq \hat{t}_0 + \tilde{t}} \Vert h(t) \Vert_{L^{\infty}} \leq \frac{1}{2} \bar{M}. 
\end{align*}
By repeating the procedure up to $T_0$, the solution $h(t)$ to the Boltzmann equation uniquely exists for $t\in [0,T_0]$ and satisfies \eqref{proof est 3}. From definition \eqref{T0} of $T_0$ and the estimate \eqref{proof est 2}, one obtains that
\begin{align*}
	\Vert h(T_0) \Vert_{L^{\infty}} \leq \frac{3}{8} \kappa + \frac{1}{4} \kappa < \kappa, 
\end{align*}
because of $D \leq \frac{\kappa}{4}$. For $t \in [T_0, \infty)$, Theorem \ref{small global} gives the global well-posedness and expoential decay property of the solution to the Boltzmann equation with $\Vert h_0 \Vert \leq \kappa $. Thus, treating $\Vert h(T_0) \Vert_{L^{\infty}}$ as initial data in Theorem \ref{small global} yields that 
\begin{align*}	
	\Vert h(t) \Vert_{L^{\infty}} \leq C e^{-\vartheta (t- T_0)} \Vert h(T_0) \Vert_{L^{\infty}}, \quad \forall t \geq T_0. 
\end{align*}
\newline 
\noindent {\bf Acknowledgement:}
G.-H. Ko and S.-J. Son are supported by the National Research Foundation of Korea (NRF)
grant funded by the Korean government (MSIT) (RS-2023-00219980 and RS-2023-00210484).
S.-J. Son is supported by the National Research Foundation of Korea (NRF) grant funded by the
Korean government (MSIT) (RS-2023-00212304).

\begin{thebibliography}{10}
	
	
	\bibitem{BKLY_BGK}
	G.-C. Bae, G.~Ko, D.~Lee, and S.-B. Yun.
	\newblock Large amplitude problem of bgk model: Relaxation to quadratic
	nonlinearity.
	\newblock {\em arXiv preprint arXiv:2301.09857}, 2023.
	
	\bibitem{BBBD}
	C.~Baranger, M.~Bisi, S.~Brull, and L.~Desvillettes.
	\newblock On the {C}hapman-{E}nskog asymptotics for a mixture of monoatomic and
	polyatomic rarefied gases.
	\newblock {\em Kinet. Relat. Models}, 11(4):821--858, 2018.
	
	\bibitem{BN1}
	N.~Bernhoff.
	\newblock Linearized {B}oltzmann collision operator: {I}. {P}olyatomic
	molecules modeled by a discrete internal energy variable and multicomponent
	mixtures.
	\newblock {\em Acta Appl. Math.}, 183:Paper No. 3, 45, 2023.
	
	\bibitem{BN2}
	N.~Bernhoff.
	\newblock Linearized {B}oltzmann collision operator: {II}. {P}olyatomic
	molecules modeled by a continuous internal energy variable.
	\newblock {\em Kinet. Relat. Models}, 16(6):828--849, 2023.
	
	\bibitem{Be}
	A.~Berroir.
	\newblock Contribution \`a la th\'{e}orie cin\'{e}tique des gaz polyatomiques.
	{I}.
	\newblock {\em Ann. Inst. H. Poincar\'{e} Sect. A (N.S.)}, 12:1--70, 1970.
	
	\bibitem{Bi}
	G.~A. Bird.
	\newblock {\em Molecular gas dynamics and the direct simulation of gas flows},
	volume~42 of {\em Oxford Engineering Science Series}.
	\newblock The Clarendon Press, Oxford University Press, New York, 1995.
	\newblock Corrected reprint of the 1994 original, With 1 IBM-PC floppy disk
	(3.5 inch; DD), Oxford Science Publications.
	
	\bibitem{BL}
	C.~Borgnakke and P.~S. Larsen.
	\newblock Statistical collision model for monte carlo simulation of polyatomic
	gas mixture.
	\newblock {\em J. Comput. Phys}, 18(4):405--420, 1975.
	
	\bibitem{BCRY}
	S.~Boscarino, S.~Y. Cho, G.~Russo, and S.-B. Yun.
	\newblock Convergence estimates of a semi-{L}agrangian scheme for the
	ellipsoidal {BGK} model for polyatomic molecules.
	\newblock {\em ESAIM Math. Model. Numer. Anal.}, 56(3):893--942, 2022.
	
	\bibitem{BDLP}
	J.-F. Bourgat, L.~Desvillettes, P.~Le~Tallec, and B.~Perthame.
	\newblock Microreversible collisions for polyatomic gases and {B}oltzmann's
	theorem.
	\newblock {\em European J. Mech. B Fluids}, 13(2):237--254, 1994.
	
	\bibitem{BS}
	S.~Brull and J.~Schneider.
	\newblock On the ellipsoidal statistical model for polyatomic gases.
	\newblock {\em Contin. Mech. Thermodyn.}, 20(8):489--508, 2009.
	
	\bibitem{BST1}
	S.~Brull, M.~Shahine, and P.~Thieullen.
	\newblock Compactness property of the linearized {B}oltzmann operator for a
	diatomic single gas model.
	\newblock {\em Netw. Heterog. Media}, 17(6):847--861, 2022.
	
	\bibitem{BST2}
	S.~Brull, M.~Shahine, and P.~Thieullen.
	\newblock Fredholm property of the linearized {B}oltzmann operator for a
	polyatomic single gas model.
	\newblock {\em Kinet. Relat. Models}, 17(2):234--252, 2024.
	
	\bibitem{Br}
	R.~{Brun}.
	\newblock {\em {Transport et relaxation dans les {\'e}coulements gazeux.}}
	\newblock 1986.
	
	\bibitem{CaoJFA}
	C.~Cao.
	\newblock Cutoff {B}oltzmann equation with polynomial perturbation near
	{M}axwellian.
	\newblock {\em J. Funct. Anal.}, 283(9):Paper No. 109641, 105, 2022.
	
	\bibitem{CKLVPB}
	Y.~Cao, C.~Kim, and D.~Lee.
	\newblock Global strong solutions of the {V}lasov-{P}oisson-{B}oltzmann system
	in bounded domains.
	\newblock {\em Arch. Ration. Mech. Anal.}, 233(3):1027--1130, 2019.
	
	\bibitem{CSC}
	S.~Chapman and T.~G. Cowling.
	\newblock {\em The {M}athematical {T}heory of {N}on-uniform {G}ases}.
	\newblock Cambridge University Press, Cambridge, 1939.
	
	\bibitem{DV}
	L.~Desvillettes and C.~Villani.
	\newblock On the trend to global equilibrium for spatially inhomogeneous
	kinetic systems: the {B}oltzmann equation.
	\newblock {\em Invent. Math.}, 159(2):245--316, 2005.
	
	\bibitem{Diperna-Lion}
	R.~J. DiPerna and P.-L. Lions.
	\newblock On the {C}auchy problem for {B}oltzmann equations: global existence
	and weak stability.
	\newblock {\em Ann. of Math. (2)}, 130(2):321--366, 1989.
	
	\bibitem{DH}
	H.~B. Drange.
	\newblock The linearized {B}oltzmann collision operator for cut-off potentials.
	\newblock {\em SIAM J. Appl. Math.}, 29(4):665--676, 1975.
	
	\bibitem{DHWY2017}
	R.~Duan, F.~Huang, Y.~Wang, and T.~Yang.
	\newblock Global well-posedness of the {B}oltzmann equation with large
	amplitude initial data.
	\newblock {\em Arch. Ration. Mech. Anal.}, 225(1):375--424, 2017.
	
	\bibitem{DKL2023}
	R.~Duan, G.~Ko, and D.~Lee.
	\newblock The {B}oltzmann equation with a class of large-amplitude initial data
	and specular reflection boundary condition.
	\newblock {\em J. Stat. Phys.}, 190(12):Paper No. 189, 46, 2023.
	
	\bibitem{DLpoly}
	R.~Duan and Z.~Li.
	\newblock Global bounded solutions to the {B}oltzmann equation for a polyatomic
	gas.
	\newblock {\em Internat. J. Math.}, 34(7):Paper No. 2350036, 43, 2023.
	
	\bibitem{DW2019}
	R.~Duan and Y.~Wang.
	\newblock The {B}oltzmann equation with large-amplitude initial data in bounded
	domains.
	\newblock {\em Adv. Math.}, 343:36--109, 2019.
	
	\bibitem{EG}
	A.~Ern and V.~Giovangigli.
	\newblock {\em Multicomponent transport algorithms}, volume~24 of {\em Lecture
		Notes in Physics. New Series m: Monographs}.
	\newblock Springer-Verlag, Berlin, 1994.
	
	\bibitem{EP}
	M.~J. Esteban and B.~Perthame.
	\newblock On the modified {E}nskog equation for elastic and inelastic
	collisions. {M}odels with spin.
	\newblock {\em Ann. Inst. H. Poincar\'{e} C Anal. Non Lin\'{e}aire},
	8(3-4):289--308, 1991.
	
	\bibitem{MBKK}
	F.R.W. McCourt, J.J.M. Beenakker, W.E. Köhler, I. Kuscer.
	\newblock {\em Nonequilibrium Phenomena in Polyatomic Gases}.
	\newblock Oxford University Presss, 1990
	
	\bibitem{GP}
	I.~M. Gamba and M.~Pavi\'{c}-\v{C}oli\'{c}.
	\newblock On the {C}auchy problem for {B}oltzmann equation modeling a
	polyatomic gas.
	\newblock {\em J. Math. Phys.}, 64(1):Paper No. 013303, 51, 2023.
	
	\bibitem{glassey}
	R.~T. Glassey.
	\newblock {\em The {C}auchy problem in kinetic theory}.
	\newblock Society for Industrial and Applied Mathematics (SIAM), Philadelphia,
	PA, 1996.
	
	\bibitem{GuoVPB}
	Y.~Guo.
	\newblock The {V}lasov-{P}oisson-{B}oltzmann system near {M}axwellians.
	\newblock {\em Comm. Pure Appl. Math.}, 55(9):1104--1135, 2002.
	
	\bibitem{GuoVMB}
	Y.~Guo.
	\newblock The {V}lasov-{M}axwell-{B}oltzmann system near {M}axwellians.
	\newblock {\em Invent. Math.}, 153(3):593--630, 2003.
	
	\bibitem{Guo10}
	Y.~Guo.
	\newblock Decay and continuity of the {B}oltzmann equation in bounded domains.
	\newblock {\em Arch. Ration. Mech. Anal.}, 197(3):713--809, 2010.
	
	\bibitem{GKTT2017}
	Y.~Guo, C.~Kim, D.~Tonon, and A.~Trescases.
	\newblock Regularity of the {B}oltzmann equation in convex domains.
	\newblock {\em Invent. Math.}, 207(1):115--290, 2017.
	
	\bibitem{HRY}
	B.-H. Hwang, T.~Ruggeri, and S.-B. Yun.
	\newblock On a relativistic {BGK} model for polyatomic gases near equilibrium.
	\newblock {\em SIAM J. Math. Anal.}, 54(3):2906--2947, 2022.
	
	\bibitem{Kim2011}
	C.~Kim.
	\newblock Formation and propagation of discontinuity for {B}oltzmann equation
	in non-convex domains.
	\newblock {\em Comm. Math. Phys.}, 308(3):641--701, 2011.
	
	\bibitem{KimLee}
	C.~Kim and D.~Lee.
	\newblock The {B}oltzmann equation with specular boundary condition in convex
	domains.
	\newblock {\em Comm. Pure Appl. Math.}, 71(3):411--504, 2018.
	
	\bibitem{KimLeeNonconvex}
	C.~Kim and D.~Lee.
	\newblock Decay of the {B}oltzmann equation with the specular boundary
	condition in non-convex cylindrical domains.
	\newblock {\em Arch. Ration. Mech. Anal.}, 230(1):49--123, 2018.
	
	\bibitem{KL_holder}
	C.~Kim and D.~Lee.
	\newblock H\"{o}lder regularity of the {B}oltzmann equation past an obstacle.
	\newblock {\em Comm. Pure Appl. Math.}, 77(4):2331--2386, 2024.
	
	\bibitem{KKL}
	G.~Ko, C.~Kim, and D.~Lee.
	\newblock Dynamical billiard and a long-time behavior of the boltzmann equation
	in general 3d toroidal domains.
	\newblock {\em arXiv preprint arXiv:2304.04530}, 2023.
	
	\bibitem{KLKRM}
	G.~Ko and D.~Lee.
	\newblock On $ c^{2} $ solution of the free-transport equation in a disk.
	\newblock {\em Kinetic and Related Models}, 16(3):311--372, 2023.
	
	\bibitem{KLP2022}
	G.~Ko, D.~Lee, and K.~Park.
	\newblock The large amplitude solution of the {B}oltzmann equation with soft
	potential.
	\newblock {\em J. Differential Equations}, 307:297--347, 2022.
	
	\bibitem{LP}
	E.~Lifshitz and L.~Pitaevski.
	\newblock {\em Physical Kinetics}.
	\newblock Oxford Engineering Science Series. Pergamon Press, Oxford.
	
	\bibitem{PY2}
	S.~J. Park and S.-B. Yun.
	\newblock Entropy production estimates for the polyatomic ellipsoidal {BGK}
	model.
	\newblock {\em Appl. Math. Lett.}, 58:26--33, 2016.
	
	\bibitem{PY1}
	S.~J. Park and S.-B. Yun.
	\newblock Cauchy problem for the ellipsoidal {BGK} model for polyatomic
	particles.
	\newblock {\em J. Differential Equations}, 266(11):7678--7708, 2019.
	
	\bibitem{SY}
	S.-j. Son and S.-B. Yun.
	\newblock The {ES}-{BGK} for the polyatomic molecules with infinite energy.
	\newblock {\em J. Stat. Phys.}, 190(8):Paper No. 129, 27, 2023.
	
	\bibitem{Ukai}
	S.~Ukai.
	\newblock On the existence of global solutions of mixed problem for non-linear
	{B}oltzmann equation.
	\newblock {\em Proc. Japan Acad.}, 50(3):179--184, 1974.
	
	\bibitem{Yun}
	S.-B. Yun.
	\newblock Ellipsoidal {BGK} model for polyatomic molecules near {M}axwellians:
	a dichotomy in the dissipation estimate.
	\newblock {\em J. Differential Equations}, 266(9):5566--5614, 2019.
	
\end{thebibliography}

\end{document}